\documentclass[a4paper]{article}
\usepackage[english]{babel}
\usepackage{xcolor}
\usepackage{enumerate}
\usepackage{amsmath,amsfonts,amsthm,amssymb,amscd,lscape,amstext}
\usepackage{booktabs}
\usepackage{tabularx}

\def\XXint#1#2#3{{\setbox0=\hbox{$#1{#2#3}{\int}$}
     \vcenter{\hbox{$#2#3$}}\kern-.5\wd0}}

\theoremstyle{plain}
\newtheorem{lemma}{Lemma}[section]
        \newtheorem{proposition}[lemma]{Proposition}
        \newtheorem{theorem}[lemma]{Theorem}
        \newtheorem{definition}{Definition}[section]
        \newtheorem{remark}[lemma]{Remark}

\newtheorem{nota}{Notation}
\theoremstyle{definition}
\numberwithin{equation}{section}

\newcommand {\R} {\mathbb{R}}
\newcommand {\N} {\mathbb{N}}
\newcommand {\C} {\mathbb{C}}
\newcommand{\diff}{\hspace{0.3em}\mathrm{d}}
\newcommand {\p} {\partial}
\newcommand{\ep}{\varepsilon}

\newcommand{\la}{\langle}
\newcommand{\ra}{\rangle}
\newcommand{\Hdistr}{H^{-1\slash2}_{00}(\Sigma)}
\newcommand{\Htrace}{H^{1\slash2}_{00}(\Sigma)}
\newcommand{\GG}{\mathbf{G}}

\newcommand{\uu}{\mathbf{u}}
\newcommand{\vv}{\mathbf{v}}
\newcommand{\ww}{\mathbf{w}}

\newcommand{\CC}{\mathcal{C}}

{\setlength{\parindent}{0pt}}

\title{ {\bf The local complex Calder\'on problem.\\ Stability in a layered medium\\ for a special type of anisotropic admittivity}}

{\setlength{\parindent}{0pt}
\author{Sonia Foschiatti \thanks{Faculty of Mathematics, University of Vienna, Austria. Email:sonia.foschiatti@univie.ac.at.}\qquad Romina Gaburro\thanks{Department of Mathematics and Statistics, University of Limerick, Ireland.  Email: romina.gaburro@ul.ie}\qquad\\ Eva Sincich\thanks{Dipartimento di Matematica, Informatica e Geoscienze, Universit\`{a} di Trieste, Italy. Email:esincich@units.it.}}}

\author{Sonia Foschiatti\thanks{Faculty of Mathematics, University of Vienna, Austria.  E-mail: \textsf{sonia.foschiatti@univie.ac.at}} \ \ \ \ Romina Gaburro\thanks{Department of Mathematics and Statistics, CONFIRM-Science Foundation Ireland, Health Research Institute (HRI), University of Limerick, Ireland.   E-mail: \textsf{romina.gaburro@ul.ie}}  \ \ \ \ Eva
Sincich\thanks{Dipartimento di Matematica e Geoscienze,
Universit\`a degli Studi di Trieste, Italy. E-mail: \textsf{esincich@units.it}} 
}

\date{}

\begin{document}

\maketitle

\begin{center}
\noindent \textbf{Abstract} 
\end{center}

 We deal with Calder\'on's problem in a layered anisotropic medium $\Omega\subset\mathbb{R}^n$, $n\geq 3$, with complex anisotropic admittivity $\sigma=\gamma A$, where $A$ is a known Lipschitz matrix-valued function. We assume that the layers of $\Omega$ are fixed and known and that $\gamma$ is an unknown affine complex-valued function on each layer. We provide H\"{o}lder and Lipschitz stability estimates of $\sigma$ in terms of an ad hoc misfit functional as well as the more classical Dirichlet to Neumann map localised on some open portion $\Sigma$ of $\partial\Omega$, respectively.

\medskip

\medskip
 

\medskip

\medskip



\section{Introduction}
We consider the inverse problem of determining the (complex) anisotropic \textit{admittivity} $\sigma$ of a domain $\Omega\subset\mathbb{R}^n$ from infinite measurements of voltage and electric current taken on a portion $\Sigma$ of the boundary of $\Omega$, $\partial\Omega$.
The electrostatic potential $u$ in $\Omega$, is governed by the equation
\begin{equation}\label{eq conduttivita'}
Lu=\mbox{div}(\sigma\nabla{u}) = 0 \qquad \mbox{in} \quad \Omega ,
\end{equation}
where the (possibly anisotropic) electric admittivity at frequency $\omega$ is given by the complex-symmetric matrix valued function 
\begin{equation}\label{complex sigma}
\sigma(x)=\sigma^r(x) + i\omega\sigma^i(x),\qquad x\in\Omega, 
\end{equation}
where $\sigma^r$ is the \textit{conductivity} of $\Omega$ and $\sigma^i$ is the \textit {permittivity} of $\Omega$. 

This problem, also known as Electrical Impedance Tomography (EIT) or Calde-r\'on's problem, arises in many different fields such as geophysics, medicine, and nondestructive testing of materials. When $\sigma$ is real ($\omega=0$, hence $\sigma =\sigma^r$), the inverse problem is known as the inverse conductivity or direct current problem and its first mathematical formulation is due to Calder\'{o}n \cite{C}, where he addressed the problem of whether it is possible to determine the (isotropic) conductivity $\sigma = \gamma I$ by the Dirichlet to Neumann (D-N) map
$$
\Lambda_{\sigma}:\,{H}^{\frac{1}{2}}(\partial\Omega)\,\ni\,u\vert_{\partial\Omega}\,
\rightarrow\,{\sigma}\nabla{u}\cdot\nu\vert_{\partial\Omega}\in{H}^{-\frac{1}{2}}(\partial\Omega), 
$$
where $\nu$ denotes the outward unit normal to the boundary $\partial\Omega$. This seminal paper opened the way to the solution to the uniqueness issue where one is asking whether the \textit{real} conductivity $\sigma$ can be determined by the knowledge of $\Lambda_{\sigma}$. As main contributions in this respect we mention the papers by Alessandrini \cite{A1}, Kohn and Vogelius \cite{Koh-V1, Koh-V2}, Nachman \cite{N}, and Sylvester and Uhlmann \cite{Sy-U}. We also refer to \cite{Bo, U} for an overview on this inverse problem.

The complex case has been less studied to date and in this respect we recall the uniqueness results of \cite{Bu, Fr} in the two-dimensional setting, and that of Lipschitz stability in \cite{BerFra11}, all concerning the isotropic case. We wish also to emphasise that a complex-valued $\sigma$ as in \eqref{complex sigma} that takes into account not only the conductivity $\sigma^r$ but also the permittivity $\sigma^i$ in $\Omega$ at a fixed frequency $\omega$, is more realistic than the real case ($\sigma = \sigma^r$), as both $\sigma^r$ and $\sigma^i$ are key to the materials' characterization within $\Omega$ through external electric fields.

Conductors have a high electrical conductivity $\sigma^r$, allowing both direct and alternating currents to flow through them, while dielectric materials have a high permittivity $\sigma^i$ and only allow for the alternating current to flow through them, therefore a complex $\sigma$ that takes into account both the $\sigma^r$ and $\sigma^i$ seems more appropriate. Moreover, while alternating currents are commonly used in geophysical and process monitoring applications, they are essential in medical imaging, where the direction of the injected current must be reversed after a short enough time from injection to avoid electrochemical effects (see \cite{AdGabLio15} and references within it). In medical applications, the injection of currents in one direction without reversion, would also result in the transport of ions, causing the stimulation of nerves. 

In EIT an oscillator is used to produce a sinusoidal current at a fixed frequency $\omega$. The response of the observed material may also vary with different frequencies. If one considers biological cells, higher frequency current might be able to penetrate their capacitive membrane, while lower frequencies currents will only be able to pass around the cells. The measurements resulting from employing different frequencies might therefore have an impact on the resulting reconstructed EIT image.

For the alternating current, a common accepted starting mathematical model is time-harmonic Maxwell equations at a fixed angular frequency $\omega$. Other inverse problems are modelled by the time-harmonic Maxwell equations, including optical tomography and scattering problems. In complex EIT, their reduction to the model \eqref{eq conduttivita'}-\eqref{complex sigma} considered here is justified by assuming that the so-called transient components of all fields are negligible (i.e., that measurements are taken after some settling time), that the effects of the magnetic field are negligible, and that $\omega$ is relatively low (see \cite{AdGabLio15} for a derivation of \eqref{eq conduttivita'}-\eqref{complex sigma}).

Another feature of the present paper is the presence of anisotropic admittivity. This choice is motivated by the fact that anisotropy is commonly present in nature, e.g. in the human body; in the theory of homogenization; as a result of deformation of an isotropic material. 

However, with the counterexample by Tartar \cite{Koh-V1} the general anisotropic case still poses several unresolved issues.  A line of research, since the seminal paper of Lee and Uhlmann \cite{Le-U}, investigates the determination of $\sigma$ modulo a change of variables which fixes the boundary (see \cite{La-U, La-U-T} ). In many applications, however, knowledge of position and, hence, coordinates (variables) are important. In this direction we refer to \cite{A}, \cite{A-G1}, \cite{A-G2}, \cite{AGS}, \cite{Koh-V1}, \cite{Lio} and the line of research initiated in \cite{AleVes05}.

Here we address the issue of stability for the inverse problem at hand. As is well known (see \cite{A}), even in the isotropic case, the optimal stability of $\sigma$ with respect to $\Lambda_{\sigma}$ or its local version $\Lambda_{\sigma}^{\Sigma}$, is logarithmic (see also \cite{B-B-R, B-F-R, Liu, Ma}). On the other hand, Lipschitz stability can be restored if $\sigma$ is deemed to belong to a finite dimensional space \cite{AleVes05} (see also \cite{AleDeHGabSin17, AleDeHGabSin18, BerFra11, FosGabSin21, G-S, RS1} and \cite{A-dH-G-S1, Be-dH-F-S, Be-dH-Q, Be-Fr-Mo-Ro-Ve, Be-Fr-V, dH-Q-S, dH-Q-S1, Fos23, RS2}) for related inverse problems. We also recall the uniqueness results of \cite{Al-dH-G, D1, D2, D3} in the case of piecewise constant conductivities. The results mentioned above only concern the real case.

Following this latter line of research, we consider the case when the conductivity is anisotropic and \textit{a-priori} known to be of type $\sigma = \gamma A$, where $A$ is a known Lipschitz continuous matrix-valued function on $\Omega$ and $\gamma$ is a piecewise-affine unknown complex-valued function on a given partition of $\Omega$ (a precise formulation is give in section \ref{sec: apriori}). Our setting is the one of \cite{AleDeHGabSin18}, in which $\Omega$ is covered by a finite family of known nested domains $\Omega_N\Subset\Omega_{N-1}\Subset \dots \Subset \Omega_1\Subset \Omega_0 = \Omega$, with $C^{1,\alpha}$ boundaries $\p\Omega_m$, for $\alpha\in(0,1]$, for $m=0,\dots , N$. Within this setting, we assume that the unknown conductivity $\sigma$ has the structure
\begin{equation}\label{a priori info su sigmaj}
\sigma(x)=\sum_{m=1}^{N+1}\gamma_{m}(x)\,\chi_{D_m}(x)\,A(x),\qquad\mbox{for any}\:x\in\Omega,
\end{equation}
where $\gamma_{m}(x)$ is an unknown affine complex scalar function on $D_m$, $A$ is a known Lipschitz continuous real matrix-valued function on $\Omega$ and $\{D_m\}_{m=1}^N$ is a given layered partition of $\Omega$ defined as follows 
\begin{eqnarray*}
D_m=\Omega_{m-1}\setminus \overline{\Omega}_m,\quad\text{for}\quad m=1,\dots,\:N\quad\text{and}\quad D_{N+1}=\Omega_N.
\end{eqnarray*}
We also assume that the \textit{jump} or \textit{visibility condition}
\begin{equation*}
\gamma_m \neq \gamma_{m+1},\quad \textnormal{for}\quad m=1,\dots , N-1,
\end{equation*}
is satisfied. In the above framework, we establish a classical Lipschitz stability estimate for $\sigma$ with respect to the local D-N map, localised on an open portion $\Sigma\subset\partial\Omega$, $\Lambda_{\sigma}^{\Sigma}$ (the precise formulation is given in \eqref{localDN}). 

Our stability result for complex anisotropic admittivities of type \eqref{a priori info su sigmaj} is also presented in terms of a misfit functional that was already introduced in the real case in \cite{FosGabSin21} (we also refer to \cite{A-dH-F-G-S} where a similar misfit functional was originally exploited in the context of an inverse problem for the Helmholtz equation). Here, as measurements are locally taken on an open portion $\Sigma\subset\partial\Omega$, we conveniently enlarge the physical domain $\Omega$ to an augmented domain $\widetilde\Omega_0$ and consider Green's functions $G_i$ for $\mbox{div}(\sigma^{(i)}\nabla\cdot)$ in $\widetilde\Omega_0$, for $i=1,2$, with poles $y,z\in\widetilde\Omega_0\setminus\bar\Omega$ respectively. The error's measurements is then expressed via the misfit functional
\begin{equation}\label{misfitfun}
\mathcal{J}(\sigma^{(1)},\sigma^{(2)})=\int_{D_y\times D_z} \left|S_0(y,z)\right|^2 \diff y \diff z,
\end{equation}
where $D_y$, $D_z \Subset (\widetilde\Omega_0\setminus\bar\Omega)$ and
\begin{equation*}
S_0(y,z) = \int_{\Sigma} \left[ G_2(\cdot,z)\sigma^{(1)}(\cdot)\nabla G_1(\cdot,y)\cdot\nu - G_1(\cdot,y)\sigma^{(2)}(\cdot)\nabla G_2(\cdot,z)\cdot\nu\right]\:\diff S.
\end{equation*}
Our stability estimate of $\sigma$ in terms of \eqref{misfitfun} is of H\"older type:
\begin{equation}\label{stabilita' globale intro}
	\|\sigma^{(1)}-\sigma^{(2)}\|_{L^{\infty}(\Omega)}\leq C
	\left(\mathcal{J}(\sigma^{(1)},\sigma^{(2)})\right)^{1\slash 2},
	\end{equation}
where $C>0$ is a constant that depends on the \textit{a-priori} information only. The augmented domain $\widetilde\Omega_0$ is chosen in such a way that $G_i(\cdot, y)\big|_{\partial\Omega}$, for $i=1,2$ is supported in $\Sigma$ in the trace sense (see \eqref{localDN}). Similarly to the real case treated in \cite{FosGabSin21}, the Cauchy data of type $\left\{G(\cdot , y)\big|_{\partial\Omega},\: \sigma\nabla G(\cdot , y)\cdot\nu \big|_{\partial\Omega}\right\}$, with $y\in\widetilde\Omega_0\setminus\bar\Omega$, is sufficient to stably determine the complex $\sigma$ treated here as well. The existence of the Green's function for $L$ as defined in \eqref{eq conduttivita'} - \eqref{a priori info su sigmaj} in the whole augmented domain $\widetilde\Omega_0$, is based on \textit{a-priori} regularity estimates for the underlying strongly elliptic system with piecewise smooth coefficient. To this end, we invoke the existence theory developed by Hofmann and Kim \cite{HofKim07} based on the condition of local H\"older continuity of the solutions to \eqref{eq conduttivita'} that guarantees the existence of the Green's function for the system in \eqref{eq conduttivita'}. In our setting, such condition is satisfied due to the \textit{a-priori} bound of Li and Niremberg \cite{LiNir03}. Details can be found in Section \ref{sec: preliminary}. In this respect we recall the geometric configuration considered in \cite{BerFra11}, where the authors constructed Green's type of functions for the forward operator in \eqref{eq conduttivita'} with poles placed far from the points where a H\"older continuity bound (for the solution) was not at hand. 

The choice in the present paper of considering anisotropic conductivities with the piecewise regularity as in \eqref{a priori info su sigmaj} on a given partition of $\Omega$, which contiguous sets arise between the interfaces of a family of nested domains, is inspired by \cite{AleDeHGabSin18}, where some of the authors of the present paper simultaneously determined, in this geometric setting, the anisotropic real piecewise constant conductivity and the interfaces of the partition of $\Omega$ by a local boundary map. Here we quantify the unique determination of the conductivity $\sigma$ in the case when $\sigma$ is complex and as in \eqref{a priori info su sigmaj}.

The paper is organized as follows. In \ref{section2} we introduce the main assumptions on the domain $\Omega$ and the anisotropic conductivity $\sigma$. Section \ref{section2} also contains the formal definition of the local D-N map (\eqref{localDN}) and the statement of our main result, the Lipschitz stability estimate in terms of the local D-N map (\ref{thm: stability estimate}). Section \ref{sec: preliminary} is devoted to the introduction of some technical tools of asymptotic estimates for the Green's function (\ref{thm: asymptotic}) and propagation of smallness (\ref{prop: uc}) needed for the machinery of the proof of \ref{thm: stability estimate}. Section \ref{sec4} is dedicated to the proof of \ref{thm: stability estimate}. In \ref{section5} we conclude with final remarks where we introduce the misfit functional and provide a H\"older stability estimate in terms of it (\ref{teorema misfit}). The Appendix contains the proofs of \ref{thm: asymptotic} and \ref{prop: uc}.

\section{Main result}\label{section2}

In this section, we state the Lipschitz stability estimate for a special class of complex conductivities with real-valued anisotropy. We begin by defining the notation, the local Dirichlet-to-Neumann map and the Alessandrini's identity. We proceed by introducing the a-priori assumptions about the domain and the coefficient. The Lipschitz stability estimate will be presented in Theorem \ref{thm: stability estimate}.
        
\setcounter{equation}{0}

\subsection{Notation and definitions}

\begin{nota}
Given $x\in \R^n$, $n\geq3$, we shall denote $x=(x',x_n)$, where $x'=(x_1,\dots,x_{n-1})\in\mathbb{R}^{n-1}$, $x_n\in\mathbb{R}$. Given $x\in \mathbb{R}^n$, we shall use the following notation for balls and cylinders:
\begin{align*}
B_r(x) &= \{y\in \R^n\,:\,|y-x|<r\},\quad B_r=B_r(0),\\
B'_r(x') &= \{y\in \R^{n-1}\,:\,|y'-x'|<r\},\quad B'_r=B'_r(0),\\
Q_{a,b}(x) &= \{(y',y_n)\,:\,|y'-x'|<a,\,|y_n-x_n|<b\},\quad Q_{a,b}=Q_{a,b}(0).
\end{align*}
Let $\R^n_{\pm} = \{ (x',x_n)\in \R^{n-1}\times \R \,:\, x_n\gtrless 0 \}$ be the positive (negative) real half space, and $B^{\pm}_r = B_r\cap \R^n_{\pm}$ be the positive (negative) semisphere centered at the origin.
\end{nota}

\begin{definition}\label{def: Holderboundary}
Let $\Omega$ be a bounded domain in $\mathbb R^n$. Given $\alpha$, $\alpha\in (0,1]$, the boundary $\partial\Omega$ is said to be of class $C^{1,\alpha}$ with positive constants $r_0$ and $M_0$, if for any point $P\in\p\Omega$ there exists a rigid transformation of coordinates under which $P=0$ and
\[
\Omega\cap B_{r_0}=\{x\in B_{r_0}\,:\,x_n>\varphi(x')\},
\]
where $\varphi$ is a $C^{1,\alpha}$  function on $B'_{r_0}$ satisfying
\begin{align*}
\varphi(0) = |\nabla_{x'} \varphi(0)| = 0,\quad
\|\varphi\|_{C^{1,\alpha}(B'_{r_0})} \leq M_0\,r_0.
\end{align*}
\end{definition}
\begin{remark}
To ensure rigour in the proofs, we adopt the convention of normalising norms. For instance, the $C^{1,\alpha}$ norm appearing above is meant as follows:
\[
\|\varphi\|_{C^{1,\alpha}(B'_{r_0})} = \|\varphi\|_{L^{\infty}(B'_{r_0})} + r_0 \|\nabla \varphi\|_{L^{\infty}(B'_{r_0})} + r_0^{1+\alpha} |\nabla \varphi|_{\alpha,B'_{r_0}},
\]
where
\[
|\nabla \varphi|_{\alpha,B'_{r_0}} = \sup_{\substack{x',y'\in B'_{r_0}\\ x'\neq y'}} \frac{|\nabla \varphi(x')-\nabla \varphi(y')|}{|x'-y'|^{\alpha}}.
\]
\end{remark}

\begin{definition}\label{def: flat portion}
Let $\Omega$ be a bounded domain in $\mathbb R^n$. A non-empty boundary portion $\Sigma$ of $\partial\Omega$ is said to be a flat portion of size $r_0>0$ if for any point $P\in\Sigma$ there exists a rigid transformation of $\R^n$ under which $P=0$ and positive real numbers $a,b<r_0$ such that
\begin{align*}
\Sigma\cap Q_{a,b} &=\{(x',x_n)\in Q_{a,b}\,:\,x_n=0\},\\ 
\Omega\cap Q_{a,b} &=\{(x',x_n)\in Q_{a,b}\,:\,x_n>0\}\\
(\R^n\setminus \Omega) \cap Q_{a,b} &=\{(x',x_n)\in Q_{a,b}\,:\,x_n<0\}.
\end{align*}
\end{definition}
\begin{nota}
For $r>0$, we denote
\begin{align}
\Omega_r(y) &= \Omega \cap B_r(y), \qquad \textnormal{for any }y\in \Omega,\label{eqn: conv1}\\
(\Omega)_r &= \{x\in \R^{n}\,:\,\textnormal{dist}(x, \p\Omega)\geq r \}\label{eqn: conv2}.
\end{align}

\end{nota} 

\begin{nota}
For any $\vv, \ww \in \C^n$, with $\vv=(v_1,\dots,v_n)$, $\ww=(w_1,\dots,w_n)$, we denote by $\vv\cdot \ww$ the holomorphic continuation of the Euclidean scalar product, which is a bilinear form without the complex conjugation of the second vector:
\[
\vv\cdot \ww = \sum_{k=1}^n v_k\,w_k.
\]
\end{nota}

\subsection{The a-priori assumptions}\label{sec: apriori}

\paragraph{Assumptions about the domain.}

We assume that $\Omega\subset \R^n$ is a bounded, measurable domain with boundary $\p\Omega$ of class $C^{1,\alpha}$ with positive constants $r_0$ and $M_0$ as per Definition \ref{def: Holderboundary}. We assume that	
\begin{equation*}
	\lvert\Omega\rvert\leq \hat{C} r_0^n,
\end{equation*}	
where $\lvert\Omega\rvert$ is the Lebesgue measure of $\Omega$ and $\hat{C}$ is a suitable positive constant. We assume that there is an open non-empty portion $\Sigma$ of $\p\Omega$ (where the measurements in terms of the DtoN map are taken) such that $\Sigma$ is flat. \\
Let $N$ be a positive integer. We assume that $\Omega_0, \Omega_1, \dots, \Omega_N$ is a finite collection of nested domains satisfying the following conditions:
\begin{enumerate}
	\item $\Omega_0=\Omega$ and $\Omega_{N+1}=\emptyset$;
	\item The nested domains satisfy the following condition:
	\[
	\Omega_N\Subset\Omega_{N-1}\Subset \dots \Subset \Omega_1\Subset \Omega_0.
	\]
	We assume that the boundaries $\p\Omega_m$ are of class $C^{1,\alpha}$ for $\alpha\in(0,1]$;
	\item We define the layers 
	\begin{align*}
		D_m &= \Omega_{m-1}\setminus \overline{\Omega}_m\quad \text{for }m=1,\dots,N,\\
		D_{N+1} &=\Omega_N,
	\end{align*}
	and we assume that these sets are connected;
	\item For each layer $D_m$, we assume that there is a non-empty flat portion $\Sigma_m$ contained in $\p\Omega_m$ of size $r_0/3$.
\end{enumerate}
\begin{remark}
	The assumption of considering flat portions serves to enhance readability of the work. Indeed, it is feasible to introduce a local diffeomorphism that flattens a $C^{1,\alpha}$ portion of the boundary. Moreover, the flat portion assumption aligns well with the assumption of a piecewise affine conductivity on the nested partition of $\Omega$ in view of the proof of the main result of this work.
\end{remark}

\paragraph{Assumptions about the anisotropic complex conductivity.}

The anisotropic complex conductivity $\sigma$ is a complex, bounded, measurable $n\times n$ matrix function of the form
\begin{equation}\label{eqn: complex-cond}
	\sigma(x) = \sigma^r(x) + i\sigma^i(x) = (\gamma^r(x) + i\gamma^i(x)) A(x),
\end{equation}
where the real part $\gamma^r(x)$ and the imaginary part $\gamma^i(x)$ are piecewise affine real functions defined on $\Omega$ as follows:
\begin{align*}
	\gamma^r(x) &= \sum_{m=1}^{N+1} \gamma^r_m(x) \chi_{D_m}(x) = \sum_{m=1}^{N+1} (s_m^r + S_m^r\cdot x)\chi_{D_m}(x), \ \ &\textnormal{for }s_m^r\in \R,\quad S_m^r\in \R^n,\\
	\gamma^i(x) &= \sum_{m=1}^{N+1} \gamma^i_m(x) \chi_{D_m}(x) = \sum_{m=1}^{N+1} (s_m^i + S_m^i\cdot x)\chi_{D_m}(x), \ \ &\textnormal{for }s_m^i \in \R,\quad S_m^i\in \R^n,
\end{align*}
where $A$ is a real $n\times n$ symmetric matrix function, and $D_m$, for $m=1,\dots,N+1$ are the nested domains. We also assume that:
\begin{itemize}
	\item[a)] There is a positive constant $\bar{\gamma}>1$ such that
	\begin{equation*}
		\bar{\gamma}^{-1}\leq \gamma^r(x),\,\,\text{and}\quad \lvert \gamma(x)\rvert \leq  \bar{\gamma} , \qquad \mbox{for any}\,x\in \Omega.
	\end{equation*}
	\item[b)] \emph{Lipschitz continuity of $A$}. We assume that $A\in C^{0,1}(\Omega,Sym_n)$, and that there exists a constant $\bar{A}>0$ such that
	\begin{equation}\label{eqn: Anorm}
		\|A\|_{C^{0,1}(\Omega,Sym_n)}\leq \bar{A}.
	\end{equation}
	\item[c)] \emph{Uniform ellipticity condition}. There exists a constant $\lambda>1$ such that
	\begin{equation}\label{eqn: elliptic condition}
		\lambda^{-1} |\xi|^2 \leq \sigma^r(x) \xi\cdot \xi \leq \lambda |\xi|^2, \qquad \mbox{for every}\, \xi\in \R^n,\, \mbox{for a. e. } x\in \Omega.
	\end{equation}
	\item[d)] \emph{Visibility condition}. For any $m\in \{1,\dots,N+1\}$,
	\[
	\gamma_{m-1}^r \neq \gamma_m^r \quad \text{and}\quad \gamma_{m-1}^i \neq \gamma_m^i.
	\]
\end{itemize}
The complex conductivity equation
\begin{equation}\label{eqn: conductivity equation}
	\mbox{div}(\sigma\nabla u) = 0,\qquad \textnormal{in }\Omega,
\end{equation}
is equivalent to the $2\times 2$ real system for the vector valued function
\begin{equation}\label{eqn: uvector}
	\uu = (u^r, u^i)^T=(u^1,u^2)^T,
\end{equation}
of the form
\begin{equation*}
	\begin{cases}
		\mbox{div}(\sigma^r\nabla u^1 - \sigma^i \nabla u^2) = 0,\\
		\mbox{div}(\sigma^i\nabla u^1 + \sigma^r \nabla u^2) = 0,
	\end{cases}
\end{equation*}
that can be written in a compact form as
\begin{equation}\label{eqn: conductivity system}
	\mbox{div}(\CC(x)\nabla \uu)=0,\qquad \textnormal{in }\Omega,
\end{equation}
where $\CC$ is the tensor
\begin{equation}\label{eqn: tensorC}
	\CC(x) =
	\begin{pmatrix}
		\sigma^r(x) &-\sigma^i(x)\\
		\sigma^i(x) &\sigma^r(x)
	\end{pmatrix}
	.
\end{equation}

In components, \eqref{eqn: conductivity system} becomes
\begin{equation*}
	\sum_{l=1}^2 \sum_{h,k=1}^n \p_{x_h}(C^{hk}_{jl}\p_{x_k}u^l)=0 \quad \text{for }j=1,2, \quad \text{in }\Omega,
\end{equation*}
applying the convention \eqref{eqn: uvector}, the tensor $\CC=\{C^{hk}_{jl}\}$ is given by
\begin{equation*}
	C^{hk}_{jl} = (\sigma^r)^{hk}\delta_{jl} - (\sigma^i)^{hk}(\delta_{j1}\delta_{l2}-\delta_{j2}\delta_{l1}).
\end{equation*}
By a straightforward calculation, by \eqref{eqn: elliptic condition}, it follows that
\[
\CC(x)\xi\cdot\xi = \sigma^r(x) \xi_1\cdot\xi_1 + \sigma^r(x) \xi_2\cdot\xi_2, \qquad \forall \xi=(\xi_1,\xi_2)\in \R^{2n},
\]
therefore $\CC$ satisfies a uniform ellipticity condition, namely for $\lambda>1$,
\begin{equation*}
	\lambda^{-1}|\xi|^2\leq \sum_{j,l=1}^2 \sum_{h,k=1}^n C^{hk}_{jl}(x) \xi_k^l\cdot \xi_h^j\leq \lambda |\xi|^2, \qquad \textnormal{for any }\xi\in \R^{2n}, \textnormal{ for a.e. }x\in\Omega.
\end{equation*}

\begin{definition}
	The set of positive constants $\{N, \hat{C}, r_0, M_0, \lambda, \bar{\gamma}, \bar{A}, n\}$ with $N\in\mathbb{N}$ and the space dimension $n\geq 3$, is called the \emph{a-priori data}.  
\end{definition}
Throughout the paper, several constants depending on the \textit{a-priori data} will appear, therefore we decide to denote them by $C, C_1,C_2\dots$, avoiding in most cases to point out their specific dependence on the a-priori data that may vary from case to case. 

\subsection{Formulation of the direct problem and local DtoN map}\label{localDN}
Consider $\Omega\subset\R^n$ a bounded domain. Let $\Sigma$ be a non-empty (flat) open portion of $\p\Omega$ of size $r_0$ as in Definition \ref{def: flat portion}. We denote with $H^1(\Omega)=H^1(\Omega,\C)$ the Sobolev space of complex-valued functions with square-integrable first order derivatives. Let
\[
H^{1/2}_{co}(\Sigma) = \left\{f\in H^{1/2}(\p\Omega,\C)\::\:\textnormal{supp}f\subset\Sigma \right\}.
\]
The trace space $H^{1/2}_{00}(\Sigma)$ is the closure of $H^{1/2}_{co}(\Sigma)$ with respect to the $H^{1/2}(\p\Omega)$-norm. The distributional trace space $H^{-1/2}_{00}(\Sigma)$ is the dual of the trace space $H^{1/2}_{00}(\Sigma)$.\\
For any $f\in \Htrace$, let $u\in H^1(\Omega)$ be the weak solution of the Dirichlet problem
\begin{equation}\label{eqn: dirichlet problem}
\begin{cases}
\ \mbox{div}(\sigma \nabla u)=0 &\textnormal{in }\,\Omega,\\
\ u=f & \textnormal{on }{\partial\Omega}.
\end{cases}
\end{equation}
The \textit{weak formulation} can be stated as follows: find a function $u\in H^1(\Omega)$ such that $u|_{\p\Omega}=f$ in the trace sense and
\begin{equation*}
\int_{\Omega} \sigma(x)\nabla u(x)\cdot \nabla \bar{\varphi}(x) \diff x = 0,\qquad \textnormal{for any }\varphi\in H^1_0(\Omega).
\end{equation*}
Let $B:H^1(\Omega) \times H^1(\Omega) \rightarrow \C$ be the sesquilinear form 
\[
B[u,v]:= \int_{\Omega} \sigma(x)\nabla u(x)\cdot \nabla \bar{v}(x) \diff x.
\]
Since $B$ is $\R$-linear in the first component, continuous, and coercive, by the Lax-Milgram theorem, we know that the weak solution $u\in H^1(\Omega)$ is unique (see \cite[pag. 261-262]{Kir05}).

\begin{definition}
The local Dirichlet-to-Neumann (DtoN) map $\Lambda^{\Sigma}_{\sigma}$ associated to the complex conductivity $\sigma$ and the flat portion $\Sigma$ is the operator
\begin{align*}
\Lambda_{\sigma}^{\Sigma}\,:\,H^{1\slash2}_{00}(\Sigma)\, &\rightarrow \,H^{-1\slash2}_{00}(\Sigma)\\
f \ \ \ \ &\mapsto \, \sigma\,\nabla u\cdot \nu\big|_{\p\Omega},\nonumber
\end{align*}
where $\nu$ is the outward unit normal of $\p \Omega$ and, for $f\in \Htrace$, $u\in{H}^{1}(\Omega)$ is the weak solution of the Dirichlet problem \eqref{eqn: dirichlet problem}.

The local DtoN map is characterized by the sesquilinear form 
\begin{equation*}
\la\Lambda_{\sigma}^{\Sigma}\:f,\:\bar{g}\ra\:=\:\int_{\:\Omega} \sigma(x)\: \nabla{u}(x)\cdot \nabla \varphi(x) \:dx,
\end{equation*}
where $g\in \Htrace$ and $\varphi\in H^{1}(\Omega)$ is any function such that $\varphi\vert_{\p\Omega}=g$ in the trace sense.  
\end{definition}

\begin{nota}
Here, $\la\cdot,\cdot\ra$ denotes the dual pairing between $\Hdistr$ and $\Htrace$ based on the $L^2(\p \Omega)$ inner product. We denote by $\parallel\cdot\parallel_{*}$ the $\mathcal{L}(\Htrace,\Hdistr)$-norm on the Banach space of
bounded linear operators from $\Htrace$ to $\Hdistr$, so that	
\[
\| \Lambda_{\sigma}^{\Sigma}\|_* = \sup \left\{ |\la\Lambda_{\sigma}^{\Sigma}\:f,\:g\ra|,\,f,g\in \Htrace, \\ \|f\|_{\Htrace} = \|g\|_{\Htrace}=1\right\}.
\]
\end{nota}

\begin{nota}
When we are dealing with two complex conductivities $\sigma^{(j)},$ $j=1,2$, we adopt the convention of writing $\Lambda_j^{\Sigma}$, to denote the corresponding DtoN map $\Lambda_{\sigma^{(j)}}^{\Sigma}$  for $j=1,2$.
\end{nota}
We introduce a well-known identity which will play a fundamental role in the derivation of the stability estimate.\\
Let $\sigma^{(1)}, \sigma^{(2)}$ be two complex conductivities satisfying the assumptions in Section \ref{sec: apriori}. Let $u_j\in H^1(\Omega)$ be the weak solution of the Dirichlet problem
\[
\textnormal{div}(\sigma^{(j)} \nabla u_j)=0 \qquad\textrm{$\textnormal{in } \Omega$},
\]
and assume that $u_j|_{\p\Omega}\in \Htrace$.
The following \emph{Alessandrini's identity} holds:
\begin{equation}\label{eqn: alessandrini}
\la (\Lambda_{1}^{\Sigma} - \Lambda_{2}^{\Sigma}) u_2|_{\p\Omega}, \bar{u}_1|_{\p\Omega} \ra = \int_{\Omega}  (\sigma^{(1)} - \sigma^{(2)}) \nabla u_1 \cdot  \nabla u_2\diff x,
\end{equation}
(see \cite[Formula (44) Appendix 6.3]{BerFra11} for a proof).

\begin{theorem}\label{thm: stability estimate}
Let $\Omega$ and $\{\Omega_m\}_{m=1}^N$ be the bounded domain and the nested domains with $C^{1,\alpha}$ class boundaries, $\alpha\in(0,1]$, respectively, satisfying the assumptions of Section \ref{sec: apriori}. Let $\sigma^{(j)}$, $j=1,2$ be two complex anisotropic conductivities of the form \eqref{eqn: complex-cond} satisfying the assumptions of Section \ref{sec: apriori}. Then there exists a positive constant $C$ depending only on the a-priori data such that
\begin{equation}\label{eqn: Lipschitz estimate}
\|\sigma^{(1)} - \sigma^{(2)} \|_{L^{\infty}(\Omega)} \leq C \|\Lambda_{1}^{\Sigma} - \Lambda_2^{\Sigma} \|_*.
\end{equation}
\end{theorem}
The proof of Theorem \ref{thm: stability estimate} is postponed to Section \ref{sec4}. In the following section we introduce some auxiliary results, concerning the existence of the Green function in the complex setting and its properties, and the characterization of the singular solutions.
	
\section{Preliminary results}\label{sec: preliminary}

\subsection{Green's functions and asymptotic estimates}\label{subsec: construction Green}

Let $\uu\in H^1(\Omega)^2=H^1(\Omega)\times H^1(\Omega)$ be a vector function as in \eqref{eqn: uvector}, and $\CC$ be the tensor as in \eqref{eqn: tensorC}. Let $L$ be a second-order elliptic operator in divergence form acting on vector valued functions in $H^1(\Omega)^2$ as follows:
\begin{equation*}
L_j \uu = \sum_{l=1}^2 L_{jl}\,u^l := \sum_{l=1}^2 \sum_{h,k=1}^n \p_{x_h}(C^{hk}_{jl} \p_{x_k} u^l)\qquad \text{for }j=1,2.
\end{equation*}
We assume that $L$ satisfies the uniform ellipticity condition with constant $\lambda>1$, namely 
\begin{equation*}
\sum_{j,l=1}^2 \sum_{h,k=1}^n C^{hk}_{jl}(x) \xi_k^l\cdot \xi_h^j \geq \lambda^{-1}\sum_{j=1}^2 \sum_{h=1}^n |\xi^j_h|^2 = \lambda^{-1} |\xi|^2 \quad \text{for }\xi=(\xi^1,\xi^2)^T\in \R^{2n},
\end{equation*}
and the uniform boundedness condition, namely
\begin{equation*}
\sum_{j,l=1}^2 \sum_{h,k=1}^n |C^{hk}_{jl}(x)|^2 \leq \lambda^2\quad \text{for any }x\in \R^n.
\end{equation*}
The adjoint operator $L^*$ of $L$ is defined as
\[
L^*_j \uu = \sum_{l=1}^2 L_{lj}\,u^l := \sum_{l=1}^2 \sum_{h,k=1}^n \p_{x_h}(C^{kh}_{lj} \p_{x_k} u^l)\qquad \text{for }j=1,2.
\]
We recall the definition of Green matrix as stated in \cite{KanKim10}.
\begin{definition}\label{def: Green matrix}
Let $\Omega\subset \R^n$ be a bounded domain. We say that the $2\times 2$ matrix valued function $\GG(x,y)$ with entries $G_{jl}(x,y)$ defined on the set $\{(x,y)\in\Omega\times\Omega: \, x\neq y\}$ is the \emph{Green matrix} of $L$ in $\Omega$ if it satisfies the following properties:
\begin{itemize}
\item[(1)] $\GG(\cdot,y)\in W^{1,1}_{loc}(\Omega)$ and $L\GG(\cdot,y)=-\delta_y I_n$ for all $y\in \Omega$ in the sense that for any $\mathbf{\phi}=(\phi^1,\phi^2)^T\in C^{\infty}_c(\Omega)^2$,
\[
\sum_{j,l=1}^2 \sum_{h,k=1}^n \int_{\Omega} C_{jl}^{hk}(x) \p_{x_k} G_{lm}(x,y)\cdot \p_{x_h} \phi^j(x) \diff x= \phi^m(y),\qquad \textnormal{for }m=1,2,
\]
where $C^{\infty}_c(\Omega)$ denotes the space of functions differentiable infinitely many times with compact support in $\Omega$.
\item[(2)] For all $y\in \Omega$, $r>0$, $\GG(\cdot,y)\in H^1(\Omega\setminus \Omega_r(y))^{2\times 2}$, where $H^1(\Omega\setminus \Omega_r(y))^{2\times 2}$ is the space of matrix functions with coefficients in $H^1(\Omega\setminus \Omega_r(y))$ using the convention \eqref{eqn: conv1}. Moreover, $\GG(\cdot,y)$ vanishes on $\p\Omega$.
\item[(3)] Given $\mathbf{f}=(f^1, f^2)^T\in L^{\infty}(\Omega)^2$, the function $\uu=(u^1, u^2)^T$ defined as
\[
u^j(y)=\int_{\Omega} [G_{j1}(x,y) f^1(x) + G_{j2}(x,y) f^2(x)] \diff x, \qquad \textnormal{for }j=1,2
\]
belongs to $H^1_{0}(\Omega)^2$ and satisfies $L^*\uu = -\mathbf{f}$ in the sense that for every $\phi=(\phi^1, \phi^2)^T\in C^{\infty}_c(\Omega)^2$,
\[
\sum_{j,l=1}^2 \sum_{h,k=1}^n \int_{\Omega} C_{jl}^{hk} \p_{x_h} u^j \cdot \p_{x_k} \phi^l = \sum_{l=1}^2 \int_{\Omega} f^l \phi^l.
\]
\end{itemize}
\end{definition}

Let us recall that in our setting, the matrix $\CC$ and the Green matrix take the form
\begin{equation}\label{eqn: greenmatrix}
\CC = 
\begin{pmatrix}
\sigma_R &-\sigma_I\\
\sigma_I &\sigma_R
\end{pmatrix}
\qquad \mbox{and} \qquad
\GG(x,y) = 
\begin{pmatrix}
G_{11}(x,y) &G_{12}(x,y)\\
G_{21}(x,y) &G_{22}(x,y)
\end{pmatrix}
,
\end{equation}
respectively. As firstly proposed in \cite[Section 4.2]{AleVes05}, we find it convenient to introduce the Green matrix, and later the Green function, not for the conductor $\Omega$ but for an enlarged domain given by adding to $\Omega$ an outer layer touching the boundary $\p\Omega$. Let $D_0$ be the outer layer contained in $\R^n\setminus \overline{\Omega}$ with $C^{1,\alpha}$ boundary with constants $M_0, r_0$ as in Definition \ref{def: Holderboundary} such that $\Sigma \subset \p\Omega\cap \bar{D}_0$. \\
The augmented domain, denoted by $\tilde \Omega_0$, is defined as
\begin{equation}\label{eqn: augmented}
\tilde \Omega_0 \ = \overset{\circ}{(\overline{\Omega\cup D_0})}.
\end{equation}
We extend the complex conductivity $\sigma$ to $D_0$ so that $A|_{D_0}=I_n$ and $\gamma|_{D_0}=1$. For simplicity, we will continue to refer to them by the same letter. We introduce a cutoff domain $(D_0)_{r_0\slash 2}$ as a subset of $D_0$ defined as
\begin{equation}\label{K00}
(D_0)_{r_0\slash 2} =\Big\{x\in D_0\,:\,\textnormal{dist}(x, \p\Omega) \geq \frac{r_0}{2}\Big\}.
\end{equation}
We denote with $\Gamma$ the fundamental solution of the Laplacian operator on $\R^n$ for $n\geq 3$ given by
\[
\Gamma(x,y) = \frac{1}{(n-2)\omega_n} |x-y|^{2-n}
\]
where $\omega_n$ is the volume of the unit ball in $\R^n$. \\
Fix $m\in\{1,\dots,N+1\}$, let $P_{m+1}\in \Sigma_{m+1}$. Without loss of generality, we consider a coordinate system where $P_{m+1}$ coincides with the origin. Set $\gamma=\gamma_m(0)$, $\tilde\gamma = \gamma_{m+1}(0)$ and $A= A(0)$. Notice that $\tilde{\gamma}, \gamma \in \C$ and $A\in Sym_n$ is a positive definite matrix.\\
We define a linear map $\tilde L:\R^n\rightarrow \R^n$ given by
\begin{align*}
\tilde L(\xi) := R\sqrt{A^{-1}}\xi,
\end{align*}
where $R$ is the planar rotation in $\R^n$ that rotates the unit vector $\frac{v}{\|v\|}$, with $v=\sqrt{A}e_n$, to the $n$th standard unit vector $e_n$. Additionally, we require that $R|_{(\pi)^{\perp}} = I_n|_{(\pi)^{\perp}},$ where $\pi$ is the plane generated by $\{e_n,v\}$ and $(\pi)^{\perp}$ is its orthogonal complement of $\pi$ in $\R^n$. Hence, the fundamental solution in $\R^n$ for the differential operator 
\[
\textnormal{div}( (\gamma + (\tilde{\gamma} - \gamma) \chi_{\R^n_+})A \nabla \cdot)
\]
has the explicit formulation given by
\begin{equation}\label{eqn: fundsol}
\displaystyle
H (x,y) = |J|
\begin{cases}
\displaystyle \frac{1}{\tilde{\gamma}}\Gamma(\tilde Lx,\tilde Ly) + \frac{\tilde{\gamma}-\gamma}{\tilde{\gamma}(\tilde{\gamma}+\gamma)}\Gamma(\tilde Lx,\tilde L^*y) &\textnormal{if }x_n, y_n >0,\\
\displaystyle \frac{2}{\gamma+\tilde{\gamma}} \Gamma(\tilde Lx,\tilde Ly) &\textnormal{if } x_n\cdot y_n<0,\\
\displaystyle \frac{1}{\gamma}\Gamma(\tilde Lx,\tilde Ly) + \frac{\gamma-\tilde{\gamma}}{\gamma(\tilde{\gamma}+\gamma)}\Gamma(\tilde Lx,\tilde L^*y) &\textnormal{if }x_n, y_n <0,
\end{cases}
\end{equation}
where $y^* = (y_1,\dots,y_{n-1},-y_n)$, $J=\sqrt{A^{-1}}$, $|J|=\textnormal{det}(\sqrt{A^{-1}})$, and the matrix $\tilde L^*=\{\tilde L^*_{i,j}\}_{i,j=1}^n$ is such that $l^*_{i,j} = l_{i,j}$ for $i=1,\dots,n-1$, $j=1,\dots,n$ and $l^*_{n,j}=-l_{n,j}$ for $j=1,\dots,n$. For more details, see \cite{FosGabSin21}.\\

In \cite[Theorem 3.6]{KanKim10}, Kang and Kim provide the existence of the Green matrix under the assumption that a H\"older continuity condition holds for the solution to \eqref{eqn: conductivity system}. In our case, the regularity is guaranteed by a result due to Li and Nirenberg \cite[Theorem 1.1]{LiNir03}. We formulate it in a general form, based on the fact that it also holds when the domain of interest intersects the discontinuity interface.

The existence of the Green matrix is given in the following theorem, which is based on the results contained in \cite{HofKim07, KanKim10}, combined with the above regularity results.
 
\begin{theorem}\label{thm: Gbounds}
Under the above assumptions, there exists a unique Green matrix $\GG(x,y)= \{G_{jl}(x,y)\}_{j,l=1}^2$ which is continuous in $\{(x,y)\in\tilde\Omega_0\times\tilde\Omega_0,\, x\neq y\}$ (see Definition \ref{def: Green matrix}). Furthermore, $\GG(x,y)$ satisfies the following estimates:
\begin{equation}\label{eqn: G1}
\|\GG(\cdot,y)\|_{H^1(\tilde\Omega_0\setminus B_r(y))} \leq C r^{1-\frac{n}{2}}, \qquad \forall \,0<r<\frac{d(y)}{2}, \quad d(y)=\text{dist}(y,\p\tilde\Omega_0),
\end{equation}
\begin{equation}\label{eqn: G2}
\|\GG(x,\cdot)\|_{H^1(\tilde\Omega_0\setminus B_r(x))} \leq C r^{1-\frac{n}{2}}, \qquad \forall\,0< r<\frac{d(x)}{2}, \quad d(x)=\text{dist}(x,\p\tilde\Omega_0),
\end{equation}
\begin{equation}\label{eqn: G3}
|\GG(x,y)|\leq C |x-y|^{2-n},\quad \text{for all }x,y\in (\tilde\Omega_0)_{r_0}, \,x\neq y,
\end{equation}
\begin{equation}\label{eqn: G4}
|\nabla \GG(x,y)|\leq C |x-y|^{1-n},\quad \text{for all }x,y\in (\tilde\Omega_0)_{r_0}, \,x\neq y,
\end{equation}
using the convention \eqref{eqn: conv2}.
\end{theorem}
\begin{proof}
The proof follows the argument presented in \cite[Proposition 5.1]{AleDiCMorRos14} in the context of elasticity.\\
In view of the results proved in \cite[Lemma 2.3]{HofKim07}, to provide the existence of the Green matrix $G$ and the properties \eqref{eqn: G1}-\eqref{eqn: G4}, it suffices to prove that there exist constants $\mu_0, C_0>0$ such that for every $R>0$ and $x\in\Omega$ with $B_{2R}(x)\subset \Omega$, all weak solutions $\uu\in H^1(B_{2R}(x))$ of the equation 
\[
\text{div}(\CC\nabla \uu)=0
\]
satisfy
\begin{equation}\label{eqn: P1}
|\uu|_{\mu_0,B_R(x)} \leq \frac{C_0}{R^{\mu_0}}\left(\frac{1}{|B_{2R}(x)|}\int_{B_{2R}(x)} |\uu|^2 \right)^{\frac{1}{2}}.
\end{equation}
Indeed, we derive \eqref{eqn: P1} with $\mu_0=1$. By \cite{LiNir03} (see also \cite[Lemma 5.2]{AleDiCMorRos14} in the elastic case), $\uu\in W^{1,\infty}(B_R(x))$, hence it is Lipschitz continuous, and
\begin{equation*}
|\uu|_{1,B_R(x)} = \|\nabla \uu\|_{L^{\infty}(B_R(x))} \leq \frac{C_0}{R^{1+\frac{n}{2}}}\left(\int_{B_{2R}(x)} |\uu|^2 \right)^{\frac{1}{2}},
\end{equation*}
from which it follows
\begin{equation}\label{eqn: P2}
\|\nabla \uu\|_{L^{\infty}(B_R(x))} \leq \frac{\tilde C_0}{R} \left(\frac{1}{|B_{2R}(x)|}\int_{B_{2R}(x)} |\uu|^2 \right)^{\frac{1}{2}}.
\end{equation}
To prove \eqref{eqn: G4}, we apply the estimate \eqref{eqn: P2} to $\GG(\cdot,y)$ in $B_s(x)$ for $s=\frac{|x-y|}{4}$, we obtain
\[
\|\nabla_x \GG(\cdot,y)\|_{L^{\infty}(B_s(x))} \leq \frac{C}{s^{1+\frac{n}{2}}} \left(\int_{B_{2s}(x)} |\GG(\xi,y)|^2 \right)^{\frac{1}{2}},
\]
and since $|\xi-y|\geq 2s$, by applying \eqref{eqn: G3} we obtain
\[
\|\nabla_x \GG(\cdot,y)\|_{L^{\infty}(B_s(x))} \leq C |x-y|^{1-n},
\]
where $C$ is a positive constant depending on $M_0, \alpha, \bar{\gamma}, \bar{A}, \lambda$.
\end{proof}
We choose to define the Green function associated to \eqref{eqn: conductivity equation} as the first column of the Green matrix $\GG$ \eqref{eqn: greenmatrix}, so that it has the form
\begin{equation}\label{eqn: complex Green}
G(x,y) := G_{11}(x,y) + i G_{21}(x,y) = G^r(x,y) + i G^i(x,y),
\end{equation}
for any $x,y \in \tilde\Omega_0$, where $G^r$ and $G^i$ denote the real and imaginary parts of the Green functions, respectively.\\
In the following theorem we give the asymptotic estimates for the Green function $G$ \eqref{eqn: complex Green} compared to $H$ \eqref{eqn: fundsol}, which will be one of the main ingredients of the proof of Theorem \ref{thm: stability estimate}.
	
\begin{theorem}\label{thm: asymptotic}
Given $m=1,\dots, N$, let $D_m$, $D_{m+1}$ be two contiguous subdomains of $\Omega$. Fix the origin at a point $P_{m+1}\in \Sigma_{m+1}$, and let $\nu$ be the outward unit normal to $\p \Omega_{m+1}$ at $P_{m+1}$. Let $Q_{m+1}$ be a point such that $Q_{m+1}\in B_{r_0/4}\cap \Sigma_{m+1}$. Then there exist positive constants $C = C(r_0, M_0, n, \bar{A}, \lambda)$, $\theta_1, \theta_2, \theta_2$, $0<\theta_1, \theta_2, \theta_3<1$ depending on the a-priori data such that the following inequalities hold true for every $\bar{x}\in B_{r_0\slash 8}\cap D_{m+1}$ and every $\bar{y}=Q_{m+1}-re_n$ for $r\in (0,r_0/8)$:
\begin{equation}\label{eqn: green1}
|G(\bar{x},\bar{y}) - H(\bar{x}, \bar{y})| \leq C |\bar{x}-\bar{y}|^{2-n+\theta_1},
\end{equation}
\begin{equation}\label{eqn: gradient green}
|\nabla_x G(\bar{x},\bar{y}) - \nabla_x H(\bar{x}, \bar{y})| \leq C  |\bar{x}-\bar{y}|^{1-n+\theta_2},
\end{equation}
\begin{equation}
|\nabla_x\nabla_y G(\bar{x},\bar{y}) - \nabla_x \nabla_y H(\bar{x}, \bar{y})| \leq C  |\bar{x}-\bar{y}|^{-n+\theta_3}.
\end{equation}
\end{theorem}

\subsection{Quantitative estimates of unique continuation}

In this subsection we introduce the notation and the parameters used for the quantitative unique continuation estimates. We define the singular solutions $S_k$, their derivatives, and we derive an upper bound for them. In Proposition \ref{prop: uc} we state the main result.\\
As a consequence of the a-priori assumptions on the scalar part of the complex conductivity, we have that there is an index $M\in \{1,\dots,N+1\}$ such that $\|\gamma^{(1)} - \gamma^{(2)}\|_{L^{\infty}(\Omega)} = \|\gamma^{(1)} - \gamma^{(2)}\|_{L^{\infty}(D_M)}$. We set
\[
E:= \|\gamma^{(1)} - \gamma^{(2)}\|_{L^{\infty}(D_M)}.\]
Let $D_0, D_1,\dots,D_M$ be a chain of contiguous domains. Set $\mathcal{U}_0 = \Omega$, and set
\begin{equation}\label{eqn: WkUk}
\mathcal{W}_k = \bigcup_{m=0}^k D_m, \qquad \mathcal{U}_k = \tilde\Omega_0\setminus \mathcal{W}_k, \qquad   \textnormal{for }k = 0,\dots, M,
\end{equation}
where $\tilde\Omega_0$ is the augmented domain defined in \eqref{eqn: augmented}.\\
Recalling our convention \eqref{eqn: complex Green}, the Green functions $G_j$ for $j=1,2$ correspond to the first column of the Green matrix \eqref{eqn: greenmatrix}.
For $y, z\in \mathcal{W}_k$, we define
\begin{equation*}
S_k(y,z) = \int_{\mathcal{U}_k} (\sigma^{(1)}(x) - \sigma^{(2)}(x))\, \nabla_x G_1(x,y)\cdot \nabla_x G_2(x,z)\,\diff x,
\end{equation*}
and for $h, l=1,\dots,n$,
\begin{equation*}
\partial_{y_h}\partial_{z_l} S_k (y,z) = \int_{\mathcal{U}_k} (\sigma^{(1)}(x) - \sigma^{(2)}(x))\, \partial_{y_h} \nabla_x G_1(x,y)\cdot \partial_{z_l} \nabla_x G_2(x,z)\,\diff x.
\end{equation*}
Thanks to Theorem \ref{thm: Gbounds}, for $y,z \in \mathcal{W}_k$ such that $\textnormal{dist}(y,\mathcal{U}_k), \textnormal{dist}(z,\mathcal{U}_k)\geq r_0/2$, we have
\begin{equation}\label{eqn: Skupper}
|S_k(y,z)| \leq C \, E \, (d(y)\,d(z))^{1-n\slash 2}, \quad \text{for }k=1,\dots,M,
\end{equation}
where $d(y)=\textnormal{dist}(y,\mathcal{U}_k), d(z)=\textnormal{dist}(z,\mathcal{U}_k)$
(see \cite[Equation (3.15)]{AleVes05}).

For $k\in \{1,\dots,M\}$, for any $y,z\in\mathcal{W}_k$, we have that $S_k(\cdot,z), S_k(y,\cdot) \in H^1_{loc}(\mathcal{W}_k)$ and they are weak solutions to the conductivity equations
\begin{equation}\label{eqn: singluar conduct}
\textnormal{div}_y(\sigma^{(2)}\nabla_y S_k(\cdot,z)) = 0,\quad \mbox{and}\quad \textnormal{div}_z(\sigma^{(1)}\nabla_z S_k(y,\cdot)) = 0, \qquad \textnormal{in } \mathcal{W}_k.
\end{equation}

In the following Proposition we introduce the quantitative estimate of unique continuation for the singular solution $S_k$ and its derivatives, whose proof is postponed to the Appendix \ref{appendix-uc}.

\begin{proposition}\label{prop: uc}
If, for some positive constant $\ep_0$,
\begin{equation*}
|S_{k}(y,z)|\leq \ep_0,\qquad \mbox{for every }\, (y,z)\in (D_0)_{r_0\slash 2}\times (D_0)_{r_0\slash 2},
\end{equation*}
with $(D_0)_{r_0\slash 2}$ as in \eqref{K00}, then there exist $\bar{r}>0$, $C_1, C_2, C_3>0$ constants depending only on the a-priori data such thta the following inequalities hold true for every $r\in(0,\bar{r}/8)$: 
\begin{align}
|S_{k}(y_{k+1},y_{k+1})| &\leq C_1 r^{-2\tilde\gamma} (\ep_0 +E) \Big(\frac{\ep_0}{\ep_0 + E}\Big)^{\tau_r^{2}\beta^{2N_1}},\label{eqn: singular 21}\\
|\partial_{y_h}\partial_{z_l}S_{k}(y_{k+1},y_{k+1})| &\leq C_2 r^{-2\tilde\gamma-2} (\ep_0 +E) \Big(\frac{\ep_0}{\ep_0 + E}\Big)^{\tau_r^{2}\beta^{2N_1}},\label{eqn: singular 22}
\end{align}
for $h,l = 1, \dots, n$, where:
\begin{itemize}
    \item $\tilde \gamma = \frac{n}{2}-1$;
    \item $y_{k+1}=P_{k+1}-r\nu(P_{k+1})$;
    \item $P_{k+1}$ is a point on $\Sigma_{k+1}$;
    \item $\nu(P_{k+1})$ is the outward unit normal of $\p \Omega_{k+1}$ at the point $P_{k+1}$.
\end{itemize}
\end{proposition}

\section{Proof of Theorem  \ref{thm: stability estimate}}\label{sec4}

\begin{proof}[Proof of Theorem \ref{thm: stability estimate}]

Let $\sigma^{(j)}$ for $j=1,2$ be two anisotropic complex conductivities satisfying the a-priori assumptions given in Section \ref{sec: apriori}, and let $\Lambda_j^{\Sigma}$ be the corresponding local DtoN maps. By \eqref{eqn: Anorm}, we derive that
\[
\|\sigma^{(1)}-\sigma^{(2)}\|_{L^{\infty}(\Omega)}\le \bar{A}\, \|\gamma^{(1)}-\gamma^{(2)}\|_{L^{\infty}(\Omega)}.
\]
Notice that there is an index $M\in \{1,\dots,N+1\}$ for which the following identity holds: 
\begin{align*}
\|\gamma^{(1)}-\gamma^{(2)}\|_{L^{\infty}(\Omega)} =\|\gamma_M^{(1)}-\gamma_M^{(2)}\|_{L^{\infty}(D_M)}.
\end{align*}
The Lipschitz stability estimate \eqref{eqn: Lipschitz estimate} follows directly from the inequality
\begin{equation}\label{eqn: gamma estimate}
\|\gamma_M^{(1)}-\gamma_M^{(2)}\|_{L^{\infty}(D_M)} \le C \|\Lambda_1^{\Sigma} - \Lambda_2^{\Sigma}\|_*,
\end{equation}
where $C>1$ is a constant depending only on the \textit{a-priori} data. We prove \eqref{eqn: gamma estimate}.\\

Notice that the norm $\|\gamma_M^{(1)}-\gamma_M^{(2)}\|_{L^{\infty}(D_M)}$ can be evaluated in terms of the quantities
\begin{align}
    &\|\gamma_M^{(1)} - \gamma_M^{(2)}\|_{L^{\infty}(\Sigma_{M}\cap B_{r_0\slash 2}(P_M))}\\
    &|\p_{\nu}(\gamma_M^{(1)} - \gamma_M^{(2)})(P_M)|,
\end{align}
where $r_0>0$ is the constant introduced in Section \ref{sec: apriori}. Indeed, introduce the following notation
\begin{equation*}
\alpha_M +\beta_M\cdot x = \big(\gamma^{(1)}_M -\gamma^{(2)}_M\big)(x),\ \ \ \textnormal{for any } x\in D_M,
\end{equation*}
with $\alpha_M \in \C, \beta_M\in \C^n$. Fix an orthonormal basis of vectors $\{e_j\}_{j=1,\dots, n-1}$ with origin at $P_M$ such that its vectors generate the hyperplane containing the flat part $\Sigma_M$. 

If we evaluate $\big(\gamma^{(1)}_M -\gamma^{(2)}_M\big)$ at the points $P_M$ and $P_M + \frac{r_0}{6} e_j$ for $j=1,\dots, n-1$ and taking their differences, it follows that
\begin{align}
    |\alpha_M + \beta_M\cdot P_M| &\leq \|\gamma_M^{(1)} - \gamma_M^{(2)}\|_{L^{\infty}(\Sigma_{M}\cap B_{r_0\slash 2}(P_M))},\\
    \sum_{j=1}^{n-1}|\beta_M \cdot e_j|
&\leq  C\|\gamma^{(1)}_M -\gamma^{(2)}_M\|_{L^{\infty}(\Sigma_M\cap B_{r_0/2}(P_M))},\\
|\beta_M\cdot\nu| &=\left|\p_{\nu}(\gamma^{(1)}_M - \gamma^{(2)}_M)(P_M)\right|.
\end{align}
In conclusion, 
\begin{equation*}
|\alpha_M| +|\beta_M|\leq C\left(\|\gamma^{(1)}_M -\gamma^{(2)}_M\|_{L^{\infty}(\Sigma_M\cap B_{r_0/2}(P_M))} + \left|\partial_{\nu}(\gamma^{(1)}_M - \gamma^{(2)}_M)(P_M)\right|\right).
\end{equation*}
\begin{nota}
Define the following quantities:
\begin{align*}
\varepsilon &= \|\Lambda_1^{\Sigma} - \Lambda_2^{\Sigma}\|_*,\\  
E &= \|\gamma^{(1)} - \gamma^{(2)}\|_{L^{\infty}(\Omega)} = \|\gamma^{(1)}_M-\gamma^{(2)}_M\|_{L^{\infty}(D_M)},\\
\delta_{M-1} &= \|\gamma^{(1)} - \gamma^{(2)}\|_{L^{\infty}(\mathcal{W}_{M-1})},
\end{align*}
where $\mathcal{W}_{M-1}$ is the set defined in \eqref{eqn: WkUk}.
\end{nota}
Let $D_1, D_2,\dots, D_M$ be the chain of subdomains satisfying the assumptions of Section \ref{sec: apriori}.

\begin{remark}
Let $\tilde\Omega_0$ be the augmented domain as defined in \eqref{eqn: augmented}. Let $\sigma^{(j)}$ for $j=1,2$ denote the extended coefficients on $\tilde\Omega_0$ with $\sigma^{(j)}|_{D_0}=I_n$ and $\gamma^{(j)}|_{D_0}=1$. The Green functions that will be considered in the proof are defined on the augmented domain.
\end{remark}
	
For $y,z\in \mathcal{W}_{M-1}$, if we choose
\[
u_1(x) = G_1(x,y) \quad \textnormal{and}\quad u_2(x)=G_2(x,z),
\]
where $G_j$ are the Green functions defined in \eqref{eqn: complex Green}, by Alessandrini's identity \eqref{eqn: alessandrini} the following identities hold:
\begin{multline}\label{eqn: Alessandrini1}
\la (\Lambda_{1}^{\Sigma} - \Lambda_{2}^{\Sigma}) G_2(x,y)|_{\p\Omega}, \bar{G}_1(x,z)|_{\p\Omega} \ra \\ = \int_{\mathcal{W}_{M-1}}  (\sigma^{(1)} - \sigma^{(2)})(x) \nabla G_1(x,y) \cdot  \nabla G_2(x,z)\diff x + S_{M-1}(y,z),
\end{multline}
and
\begin{multline*}
\la (\Lambda_{1}^{\Sigma} - \Lambda_{2}^{\Sigma}) \p_{y_n} G_2(\cdot,y)|_{\p\Omega}, \p_{z_n} \bar{G}_1(\cdot,z)|_{\p\Omega} \ra \\ = \int_{\mathcal{W}_{M-1}} (\sigma^{(1)} - \sigma^{(2)})(x) \nabla_x \p_{y_n} G_1(x,y) \cdot \nabla_x \p_{z_n} G_2(x,z) \diff x
+ \p_{y_n} \p_{z_n} S_{M-1}(y,z).
\end{multline*}
Recall that $S_{M-1}(\cdot,z)$, $S_{M-1}(y,\cdot)$, are locally weak solutions to
\begin{equation*}
\mbox{div}_y(\sigma^{(1)}\nabla_y S_{M-1}(\cdot,z)) = 0,\quad \mbox{and}\quad \mbox{div}_z(\sigma^{(2)}\nabla_z S_{M-1}(y,\cdot)) = 0, \qquad \mbox{in }\mathcal{W}_{M-1}.
\end{equation*}        
From \eqref{eqn: Alessandrini1}, for any $y,z\in (D_0)_{r_0\slash 2}$, we derive
\begin{equation}\label{eqn: S0}
|S_{M-1}(y,z)| \leq  C \:(\ep + \delta_{M-1})\,r_0^{2-n},
\end{equation}
where $C$ is a positive constant depending on the a-priori data.

Let $\rho=r_0/4$ and fix $r\in(0,r_0/2)$. Set 
\begin{equation*}
w_{M} = P_M + r \nu(P_M),
\end{equation*}
where $\nu(P_M)$ the outward unit normal of $\p \Omega_M$ at the point $P_M$.

Choose $y=z=w_M$, one can reformulate the expression of the singular solution as the sum of two integrals $I_1(w_M)$ and $I_2(w_M)$ over the domains of integration $B_{\rho}(P_M)\cap D_M$ and $\Omega\setminus (B_{\rho}(P_M)\cap D_M)$:
\begin{equation}\label{eqn: split}
S_{M-1}(w_M,w_M)=I_1(w_M)+I_2(w_M),
\end{equation}
where
\begin{align}
I_1(w_M) &= \int_{B_{\rho}(P_M)\cap D_M}(\gamma_M^{(1)}-\gamma_M^{(2)})(x)A(x)\nabla G_1(x,w_M)\cdot\nabla G_2(x,w_M)\,\diff x,\label{eqn: I11}\\
I_2(w_M) &= \int_{\mathcal{U}_{M-1}\setminus (B_{\rho}(P_M)\cap	D_M)}(\sigma^{(1)}-\sigma^{(2)})(x) \nabla G_1(x,w_M)\cdot\nabla G_2(x,w_M)\,\diff x.\label{eqn: I21}
\end{align}
The integral \eqref{eqn: I21} can be estimated by a Caccioppoli type inequality \cite[Proposition 6.1]{BerFra11} as
\begin{equation}\label{eqn: caccio}
|I_2(w_M)|\leq C\:E\:\rho^{2-n}.
\end{equation}
To estimate \eqref{eqn: I11} from below in terms of $\|\gamma^{(1)}_M -\gamma^{(2)}_M\|_{L^{\infty}(\Sigma_M\cap B_{r_0/2}(P_M))}$, notice that there exists a point $x_M\in \overline{\Sigma_M\cap B_{r_0/2}(P_M)}$ such that
\[
(\gamma^{(1)}_M-\gamma^{(2)}_M)(x_M)=\|\gamma^{(1)}_M - \gamma^{(2)}_M\|_{L^{\infty}(\Sigma_M\cap B_{r_0/2}(P_M))}.
\]
Up to a change of coordinate, we can assume that $P_M$ coincides with the origin $O$, by the triangle inequality and \eqref{eqn: Anorm} we derive
\begin{equation}\label{eqn: I1intermed}
\begin{split}
|I_1(w_M)| &\geq \left|\int_{B_{\rho}(P_M)\cap D_M}  (\gamma^{(1)}_M - \gamma^{(2)}_M)(x_M) A(x)\nabla G_1(x,w_M)\cdot \nabla G_2(x,w_M) \diff x\right|\\
&- \bar{A}\int_{B_{\rho}(P_M)\cap D_M}  |\beta_M \cdot (x-x_M)| |\nabla G_1(x,w_M)\cdot \nabla G_2(x,w_M)| \diff x.
\end{split}
\end{equation}
After adding and subtracting the fundamental solution $H_j(x,w_1)$ for $j=1,2$ in \eqref{eqn: I1intermed}, and applying the asymptotic estimate for the Green functions \eqref{eqn: gradient green}, one can bound $|I_1(w_M)|$ from below in terms of the parameter $\tau$ as follows:
\begin{equation}\label{eqn: I1lower}
|I_1(w_M)| \geq \|\gamma^{(1)}_M - \gamma^{(2)}_M\|_{L^{\infty}(\Sigma_M\cap B_{r_0/2}(P_M))}C\:r^{2-n} -C\,E\,r^{2-n+\theta_2}.
\end{equation}
By \eqref{eqn: split}, \eqref{eqn: S0}, \eqref{eqn: caccio}, and \eqref{eqn: singular 21}, it follows that
\begin{equation}\label{eqn: I1upper}
\begin{split}
|I_1(w_M)|&\leq |S_{M-1}(w_M,w_M)| + |I_2(w_M)|\\
&\leq C_1 r^{-2\tilde\gamma} (\ep + \delta_{M-1} +E) \Big(\frac{\ep + \delta_{M-1}}{\ep + \delta_{M-1} + E}\Big)^{\tau_r^{2}\beta^{2N_1}} + C\:E\:\rho^{2-n}.
\end{split}
\end{equation}
Combining \eqref{eqn: I1lower} and \eqref{eqn: I1upper} yields
\begin{multline*}
\|\gamma^{(1)}_M-\gamma^{(2)}_M\|_{L^{\infty}(\Sigma_M\cap B_{r_0/2}(P_M))} r^{2-n}\leq C_1 r^{-2\tilde\gamma} (\ep + \delta_{M-1} +E) \Big(\frac{\ep + \delta_{M-1}}{\ep + \delta_{M-1} + E}\Big)^{\tau_r^{2}\beta^{2N_1}}\\ + C\:E\:\rho^{2-n} + C\, E\,r^{2-n+\theta_2}.
\end{multline*}
Multiplying by $r^{n-2}$, one obtains
\begin{equation*}
\|\gamma^{(1)}_M - \gamma^{(2)}_M\|_{L^{\infty}(\Sigma_M\cap B_{r_0/2}(P_M))}\leq C(\ep + \delta_{M-1} + E) \bigg[\big(\frac{\ep + \delta_{M-1}}{\ep + \delta_{M-1} + E} \big)^{\tau_r^{2}\beta^{2N_1}} + r^{\theta_2} \bigg].
\end{equation*}
Following the argument of \cite{BerFra11}, by taking
\[
r=\bigg(\ln\big(\frac{\ep + \delta_{M-1}}{\ep + \delta_{M-1} + E} \big)\bigg)^{-1/4}
\]
we obtain
\begin{equation}\label{eqn: gammasurf}
    \|\gamma^{(1)}_M - \gamma^{(2)}_M\|_{L^{\infty}(\Sigma_M\cap B_{r_0/2}(P_M))}\leq C\,(\ep + \delta_{M-1} + E) \bigg(\ln\big(\frac{\ep + \delta_{M-1}}{\ep + \delta_{M-1} + E} \big)\bigg)^{-\theta_2/4}.
\end{equation}

The estimate of $|\p_{\nu}(\gamma^{(1)}_M-\gamma^{(2)}_M)(P_M)|$ can be derived in a similar fashion. Indeed, let $\rho, r, P_M, w_M$ be the quantities defined above. Consider
\[
\p_{y_n}\p_{z_n}S_{M-1}(w_M,w_M) = I_1(w_M) + I_2(w_M),
\]
where
\begin{align*}
    I_1(w_M) &= \int_{B_{\rho}(P_M)\cap D_M} (\gamma^{(1)}_M - \gamma^{(2)}_M)(x)\:A(x)\:\nabla_x \p_{y_n} G_1(x,w_M) \cdot \nabla_x \p_{z_n} G_2(x,w_M)\:\diff x,\\
    I_2(w_M) &= \int_{\mathcal{U}_{M-1} \setminus (B_{\rho}(P_M)\cap D_M)} (\sigma^{(1)} - \sigma^{(2)})(x)\:\nabla_x \p_{y_n} G_1(x,w_M)\cdot \nabla_x \p_{z_n} G_2(x,w_M)\:\diff x.
\end{align*}
Similarly as in \eqref{eqn: caccio}, one derives
\begin{equation*}
|I_2(w_M)| \leq C \,E\:\rho^{-n}.
\end{equation*}
To estimate $I_1$, due to the piecewise affine nature of the scalar part of the complex anisotropic conductivity, one applies the following identity: 
\begin{multline*}
(\gamma^{(1)}_M-\gamma^{(2)}_M)(x)=(\gamma^{(1)}_M-\gamma^{(2)}_M)(P_M) + (D_T(\gamma^{(1)}_M -\gamma^{(2)}_M)(P_M))\cdot (x-P_M)'\\ + (\partial_{\nu}(\gamma^{(1)}_M-\gamma^{(2)}_M)(P_M))(x-P_M)_n,
\end{multline*}
where $D_T$ is the tangential derivative along the directions of the orthonormal basis $\{e_j\}_{j=1,\dots,n-1}$ introduced at beginning of the proof. Hence, proceeding in a similar fashion as for \eqref{eqn: I1lower} and \eqref{eqn: I1upper}, one derives
\begin{multline*}
|\partial_{\nu}(\gamma^{(1)}_M-\gamma^{(2)}_M)(P_M)| \le  C (\ep + \delta_{M-1} + E) \bigg[\big(\frac{\ep + \delta_{M-1}}{\ep + \delta_{M-1} + E} \big)^{\tau_r^{2}\beta^{2N_1}}\\ + \bigg(\ln\big(\frac{\ep + \delta_{M-1}}{\ep + \delta_{M-1} + E} \big)\bigg)^{-\theta_2/4} + r^{\theta_3} \bigg].
\end{multline*}
Optimising the right hand side with respect to $r$ yields
\begin{equation}\label{eqn: gammanu}
|\partial_{\nu}(\gamma^{(1)}_M - \gamma^{(2)}_M) (P_M)| \le C\,(\ep + \delta_{M-1} + E) \bigg(\ln\big(\frac{\ep + \delta_{M-1}}{\ep + \delta_{M-1} + E} \big)\bigg)^{-\theta_3/(\theta_3+3)}.
\end{equation}
In conclusion, collecting \eqref{eqn: gammasurf} and \eqref{eqn: gammanu} yields
\[
E = \delta_M \leq C\,(\ep + \delta_{M-1} + E) \bigg(\ln\big(\frac{\ep + \delta_{M-1}}{\ep + \delta_{M-1} + E} \big)\bigg)^{-b},
\]
for a suitable constant $b\in (0,1)$.

Before completing the proof, we introduce the concave, non-decreasing function $\omega(t)$ on $(0,+\infty)$ as
\begin{equation}\label{eqn: omega}
\omega(t)=
\begin{cases}
    |\ln t|^{-b} &t\in (0,e^{-2}),\\
    2^{-b} &t\in[e^{-2},+\infty).
\end{cases}
\end{equation}
Recalling \cite[Equations (4.34) and (4.35)]{AleVes05}), we have that
\begin{equation}\label{property1}
(0,+\infty)\ni t \rightarrow \ t\omega\left(\frac{1}{t}\right)\quad\mbox{is a non-decreasing function}.
\end{equation}
Furthermore, we denote by $\omega^{(j)}$, $j\in \N$, the functions defined by $\omega^{(1)}=\omega,$  $\omega^{(j)}=\omega\circ \omega^{(j-1)}$, and $\omega^{(0)}(t)=t^{\theta_3/4}$ for $0<t<1$.

Applying the definition of $\omega$ \eqref{eqn: omega} yields
\begin{equation*}
E \le C\,(\ep + \delta_{M-1} + E) \omega\bigg(\frac{\ep + \delta_{M-1}}{\ep + \delta_{M-1} + E} \bigg).
\end{equation*}
By \eqref{property1}, we derive
\begin{equation}\label{eqn: semifinal}
    E \leq C (\ep + E)\,\omega^{(M-1)}\bigg(\frac{\ep}{\ep + E} \bigg).
\end{equation}
If $E>e^{2}\ep$ (the other case is trivial), \eqref{eqn: semifinal} can be written as
\[
\frac{1}{C} \leq \omega^{(M-1)}\bigg(\frac{\ep}{E} \bigg).
\]
with $C>1$, and since the function $\omega^{(M)}$ is invertible, it follows that
\[
E\leq \frac{1}{\omega^{(-(M-1)}(1/C)}\:\varepsilon,
\]
where $\omega^{(-(M-1))}$ is the inverse of $\omega^{(M-1)}$. Hence, the Lipschitz stability estimate \eqref{eqn: Lipschitz estimate} is satisfied for constants $C$ that depend on the \textit{a-priori} data only.
\end{proof}

\section{Concluding remarks}\label{section5}

\setcounter{equation}{0}

The stability estimate provided in \ref{thm: stability estimate} is a promising result for future numerical implementations. However, the numerical computation of the norm of the local Dirichlet-to-Neumann map is not always feasible. In recent years, the growing interest in computing cost functionals to determine some unknown parameters of a medium has led to the idea of expressing stability estimates in different forms. Following the approach introduced in \cite{A-dH-F-G-S} for reconstructing the wave speed of a material from Cauchy data, and in \cite{FosGabSin21} for the conductivity problem, we introduce a misfit functional and present a stability estimate formulated in terms of it as well.\\
The assumptions about the nested domains $\{\Omega_m\}_{m=0,\dots,N}$, $\sigma^{(j)}$ and its extension to $\tilde\Omega_0$, for $j=1,2$ are those of \ref{sec: apriori}.  Let $G_j$ be the complex Green's function of \eqref{def: Green matrix} associated to the elliptic operator $\mbox{div}(\sigma^{(j)}(\cdot)\nabla\cdot)$ in $\tilde\Omega_0$ for $j=1,2$\\
For $(y,z)\in D_y\times D_z\subset (D_0)_{r_0\slash 2}\times (D_0)_{r_0\slash 2}$, we define the \textit{misfit functional} as 
\begin{equation}\label{misfit2}
\mathcal{J}(\sigma^{(1)},\sigma^{(2)})=\int_{D_y\times D_z} \left|S_0(y,z)\right|^2 \diff y\, \diff z,
\end{equation}
where
\[
S_0(y,z) = \int_{\Sigma} \left[ G_2(\cdot,z)\sigma^{(1)}(\cdot)\nabla G_1(\cdot,y)\cdot\nu - G_1(\cdot,y)\sigma^{(2)}(\cdot)\nabla G_2(\cdot,z)\cdot\nu\right]\:\diff S.
\]

As in \cite[Theorem 2.1]{FosGabSin21}, one can derive a stability estimate in terms of the misfit functional \eqref{misfit2}.

\begin{theorem}\label{teorema misfit}
Let $\Omega$ be as above and $\{\Omega_m\}_{m=1,\dots,N+1}$ be a family of nested domains satisfying the assumptions of \ref{sec: apriori}. Let $\sigma^{(j)}$ for $j=1,2$ be two complex conductivities as in \eqref{eqn: complex-cond} satisfying the assumptions of \ref{sec: apriori}. Then there exists a positive constant $C$ depending only on the a-priori data such that 
\begin{equation*}
\|\sigma^{(1)}-\sigma^{(2)}\|_{L^{\infty}(\Omega)}\leq C
\left(\mathcal{J}(\sigma^{(1)},\sigma^{(2)})\right)^{1\slash 2}.
\end{equation*}
\end{theorem}

The proof of \ref{teorema misfit} follows the lines of the proofs of \ref{thm: stability estimate} and \cite[Theorem 2.1]{FosGabSin21} with the following observations. For $(y,z)\in D_y\times D_z$, by applying Alessandrini's identity \eqref{eqn: alessandrini} with $u_1(x)=G_1(x,y)$ and $u_2(x)=G_2(x,z)$, we have:
\begin{equation}\label{eqn: alessandrini2}
\la (\Lambda^{\Sigma}_1 - \Lambda^{\Sigma}_2) G_2(\cdot,z)|_{\partial\Omega}, \bar{G}_1(\cdot,y)|_{\partial\Omega} \ra = \int_{\Omega} (\sigma^{(1)}-\sigma^{(2)})(x) \nabla G_1(x,y)\cdot \nabla G_2(x,z) \diff x.
\end{equation}
Due to the regularity of the Green's functions, \eqref{eqn: alessandrini2} yields
\[
S_0(y,z) = \int_{\Omega} (\sigma^{(1)}-\sigma^{(2)})(x) \nabla G_1(x,y)\cdot \nabla G_2(x,z) \diff x.
\]
Since $S_0(\cdot,z)$ and $S_0(y,\cdot)$ are weak solutions to \eqref{eqn: singluar conduct}, we can use a local boundedness result for weak solutions of a uniformly elliptic operator to bound the supremum of $S_0$ by its $L^2$-norm as follows:
\begin{equation*}
\sup_{(y,z)\in (D_y)_{r_0/4}\times (D_z)_{r_0/4}} |S_0(y,z)| \leq C r_0^{-n} \left(\int_{D_y\times D_z} |S_0(y,z)|^2 \diff y\diff z\right)^{1/2},
\end{equation*}
where the constant $C$ depends on $n$, $\lambda$ and $|\Omega|$. Hence, by setting $\ep=\mathcal{J}(\sigma^{(1)},\sigma^{(2)})$, one can proceed as in \ref{thm: stability estimate} with minor adaptations. 

\section{Appendix}
\setcounter{equation}{0}
We prove here  Theorem \ref{thm: asymptotic} and Proposition \ref{prop: uc}.
\subsection{Proof of Theorem \ref{thm: asymptotic}}
For $m\in \{1,\dots,N+1\}$, consider $D_m, D_{m+1}$ that share a flat portion $\Sigma_{m+1}$. Moreover, we assume that there is a point $P_{m+1}\in \Sigma_{m+1}$ that under a suitable rigid transformation, coincides with the origin. Let $\nu$ be the outward unit normal at $P_{m+1}$ at the boundary $\p\Omega_{m+1}$, so that $\nu = -e_n$. Let $y=-re_n$ for some $r\in (-r_0/2,0)$. 

Let $H$ be the fundamental solution associated to the operator div$(\sigma_0\nabla \cdot)$ in $\R^n$, where 
\[
\sigma_0 (x) = (\gamma_{m}(0) + (\gamma_{m+1}(0) - \gamma_m(0))\chi_{{\R^n_+}}(x)) A(0) = \gamma_0(x) A(0).
\]

Let $\tilde\Omega_0$ be the enlarged domain defined in \eqref{eqn: augmented}. For any $y\in\tilde\Omega_0$, $G(\cdot,y)$ as a Green function is the weak solution to the problem
\begin{equation*}
\begin{cases}
	\textnormal{div}(\sigma\nabla G(\cdot,y)) = - \delta_y, \qquad &\textnormal{in }\tilde\Omega_0,\\
	G(\cdot,y)\vert_{\p\tilde\Omega_0} = 0.
\end{cases}
\end{equation*}
Define
\[
R(x,y) = G(x,y) - H(x,y).
\]
The proof of \eqref{eqn: green1} follows by estimating $|R(x,y)|$. First, we proceed by reformulating the expression of $R(x,y)$ in terms of an integral representation. Notice that for any point $y\in \tilde\Omega_0$, $R(\cdot,y)$ is a weak solution to the Dirichlet problem
\begin{equation*}
\begin{cases}
	\textnormal{div}(\sigma\nabla R(\cdot,y)) = - \textnormal{div}((\sigma - \sigma_0)\nabla H(\cdot,y)), \qquad &\textnormal{in }\tilde\Omega_0,\\
	R(\cdot,y)\vert_{\p\tilde\Omega_0} = - H(\cdot,y)\vert_{\p\tilde\Omega_0}.
\end{cases}
\end{equation*}
By the representation formula on $\tilde\Omega_0$, one derives
\begin{multline*}
R(x,y) = -\int_{\tilde\Omega_0} (\sigma(z) - \sigma_0(z)) \nabla_z H(z,y)\cdot \nabla_z \bar{G}(z,x) \diff z\\ + \int_{\p\tilde\Omega_0} \sigma(z) \nabla_z G(z,x)\cdot \nu\, \bar{H}(z,y) \diff S(z).
\end{multline*}
Regarding the integral over the $(n-1)$-dimensional space $\p\tilde\Omega_0$, we apply Theorem \ref{thm: Gbounds} to derive the following upper bound in terms of a positive constant that depends on the a-priori data only:
\begin{multline*}
\left| \int_{\p\tilde\Omega_0} \sigma(z) \nabla_z G(z,x) \cdot\nu \, \bar{H}(z,y) \diff S(z)\right| \leq C\,\|\p_{\nu} G(\cdot,x)\|_{\Hdistr}\,\,\|H(\cdot,y)\|_{\Htrace}\\
\leq C\,\|G(\cdot,x)\|_{H^1(\tilde\Omega_0\setminus B_r(x))}\,\,\|H(\cdot,y)\|_{H^1(\tilde\Omega_0\setminus B_r(y))} \leq C.
\end{multline*}
Regarding the volume integral, we observe that, by the assumption on the complex conductivity,
\begin{equation}\label{eqn: sigma0}
|\sigma(z) - \sigma_0(z)| \leq |\gamma(z)|\,|A(z)-A(0)| + |\gamma(z) - \gamma_{0}(z)| |A(0)| \leq C\,|z|.
\end{equation}
By Theorem \ref{thm: Gbounds} we derive the following pointwise bounds:	
\begin{equation*}
|\nabla G(z, x)|\le C |z - x|^{1-n}, \qquad|\nabla H(z, y)|\le C |z - y|^{1-n},
\end{equation*}
which, together with \eqref{eqn: sigma0}, leads to
\begin{equation}\label{eqn: estimate R}
\left|\int_{\tilde\Omega_0} (\sigma(z)-\sigma_0(z)) \nabla H(z,y)\cdot\nabla \bar{G} (z,x)\,\diff z\right| \le {C_1} |x-y|^{2-n+\theta_1},
\end{equation}
for some $0<\theta_1<1$. The exponent $\theta_1$ is due to the estimate of the integral in \eqref{eqn: estimate R} and depends on the dimension $n$. Indeed, on the one hand, when $n=3$, \eqref{eqn: estimate R} behaves like a logarithm $\ln(1\slash |x-y|)\sim \ln(1/h)$, and for $h<1$ can be asymptotically bounded by $1\slash h^{\theta_1}$ for $\theta_1\in (0,1)$. On the other hand, when $n>3$, the upper bound has the behaviour of the function $h^{3-n}$. In conclusion, for $x\in B^+_{r_0/2}$, $y=-r e_n$ with $r\in(-r_0/2,0)$,
\begin{equation*}
|R(x,y)| \le C |x-y|^{2-n+\theta_1}.
\end{equation*}
The estimates of $|\nabla _x R(x,y)|$ and for $|\nabla_y \nabla_x R(x,y)|$ follow the lines of the proof of \cite[Theorem 3.2]{AleDeHGabSin17} \cite[Theorem 3.1]{FosGabSin21} with a minor adaptation. To estimate $|\nabla _x R(x,y)|$, for example, we define the cylinder $Q=B'_{h/4}(x')\times (x_n,x_n+h/4)$, where $h=|x-y|$. By applying \cite{LiNir03}, we derive the following bounds:
\[
|\nabla G(\cdot,x)|_{\alpha',Q}, |\nabla H(\cdot,y)|_{\alpha',Q} \leq C h^{1-n-\alpha'}.
\]
From the following interpolation inequality 
\[
\|\nabla R(\cdot,y)\|_{L^{\infty}(Q)} \leq C \left(\|R(\cdot,y)\|_{L^{\infty}(Q)}^{\frac{\alpha'}{1+\alpha'}}\cdot |\nabla R(\cdot,y)|_{\alpha',Q}^{\frac{\alpha'}{1+\alpha'}} + \frac{1}{h}\|R(\cdot,y)\|_{L^{\infty}(Q)} \right),
\]
one derives the upper bound
\[
|\nabla R(x,y)|\leq C h^{1-n+\theta_2}, \quad \theta_2\in (0,1) \textnormal{ depends on } \alpha', \alpha.
\]

\subsection{Proof of Proposition \ref{prop: uc}}\label{appendix-uc}
In this section, we provide the quantitative estimate of unique continuation (Proposition \ref{prop: uc}). Our proof is built upon the recent derivation of the three-sphere inequality and propagation of smallness for second-order elliptic equations with piecewise complex-valued Lipschitz coefficients derived in \cite{FraVesWan23}. \\
For completeness, we first recall the three-sphere inequality that holds for balls far from the discontinuity interface in Proposition \ref{prop: 3spherefar}, which corresponds to  \cite[Theorem 3]{CarNguWan20}) and \cite[Proposition 4.1]{FraVesWan23}, and then the three inequality for the general case, Proposition \ref{prop: 3spheregeneral}, which corresponds to \cite[Proposition 4.4]{FraVesWan23}. 

Let $\Omega$ be a bounded domain with Lipschitz boundary with constants $r_0, M_0$. Let $\Sigma$ be a $C^{1,1}$ hyperplane with constants $\rho_0, K_0$ such that $\Omega\setminus \Sigma$ has two connected components, $\Omega^+$ and $\Omega^-$. Let $\sigma(x)=\{\sigma_{hl}(x)\}_{h,l=1}^n$ be a Lipschitz symmetric matrix-valued function, and assume that it has a jump at $\Sigma$. Define
\[
\sigma_+(x):=\sigma(x)|_{\Omega^+} \quad \text{and}\quad \sigma_-(x)=\sigma(x)|_{\Omega^-}.
\]
Let
\[
\sigma_{\pm}(x)= \sigma^r_{\pm}(x) + i\sigma^i_{\pm}(x),
\]
and assume that there exists a constant $\lambda>1$ such that
\[
\lambda^{-1} |\xi|^2 \leq \sigma^r_{\pm}(x) \xi\cdot\xi \leq \lambda |\xi|^2, \qquad \textnormal{for all }x\in \Omega, \,\xi\in \R^n.
\]
and
\[
\lambda^{-1} |\xi|^2 \leq \sigma^i_{\pm}(x) \xi\cdot\xi \leq \lambda |\xi|^2, \qquad \textnormal{for all }x\in \Omega, \,\xi\in \R^n.
\]
The Lipschitz condition tells us that there exists $M_1>0$ such that 
\[
|\sigma_{\pm}(x) - \sigma_{\pm}(y)|\leq M_1 |x-y|, \qquad \textnormal{for any }x,y\in \Omega.
\]
Let $v\in H^1(\Omega)$ be a weak solution to
\begin{equation*}
\mbox{div}(\sigma(x)\nabla v(x)) = 0, \qquad \textnormal{for }x\in \Omega.
\end{equation*}
Let $U\subset \Omega$ be an open bounded domain such that the coefficient $\sigma$ is Lipschitz continuous in $U$ without jumps. The following proposition corresponds to \cite[Proposition 4.1]{FraVesWan23}
\begin{proposition}\label{prop: 3spherefar}
Under the above assumptions, there exist two positive constants, $R=R(n,\lambda,M_1)$ and $s=s(\lambda)$ such that if $r_1, r_2, r_3$ are positive real numbers satisfying $r_1<r_2<\lambda r_3\slash 2< \sqrt{\lambda} R\slash 2$ with $B_{r_3}(x_0)\subset U$ for some $x_0\in U$, we have
\begin{equation*}
\|v\|_{L^2(B_{r_2}(x_0))} \leq C \|v\|^{\delta}_{L^2(B_{r_1}(x_0))} \|v\|^{1-\delta}_{L^2(B_{r_3}(x_0))}
\end{equation*}
where $C$ is explicitly given by
\[
C=e^{C_1 (r_1^{-s}-r_2^{-s})}, \quad C_1=C_1(\lambda,M_1)>1,
\]
and 
\[
\delta = \frac{(2r_2/\lambda)^{-s}-r_3^{-s}}{r_1^{-s}-r_3^{-s}}.
\]
\end{proposition}
Now we introduce the three-sphere inequality for the general case, which is an adaptation of \cite[Proposition 4.4]{FraVesWan23} to our setting.

Choose an index $m\in \{1,\dots,N\}$ and consider the domains $D_{m}, D_{m+1}$ defined in Section \ref{sec: apriori}. Set
$$\Omega=\overset{\circ}{(\overline{D_{m}\cup D_{m+1}})},$$
and let
$$\sigma_-(x) = \sigma_{m}(x),\quad \sigma_+(x) = \sigma_{m+1}(x)$$ 
be the Lipschitz continuous matrix-valued coefficients that satisfy the assumptions of Section \ref{sec: apriori}. Let $\Sigma = \p \Omega_{m+1}$ be the $C^{1,1}$ interface with constant $r_0, K_0$. Let $v\in H^1(\Omega)$ be a weak solution to
\begin{equation}\label{eqn: cond complex}
\mbox{div}((\sigma_-(x) + (\sigma_+(x)-\sigma_-(x)) \chi_{D_{m+1}}(x))\nabla v(x)) = 0, \qquad \mbox{in }\Omega.
\end{equation}

\begin{proposition}\label{prop: 3spheregeneral}
Under the above assumptions, there exists $\bar{r}>0$ depending on $r_0, K_0$ such that if $0<r_1<r_2<r_3<\bar{r}$, for $Q\in \Omega$ such that $\textnormal{dist}(Q,\p \Omega)>r_3$, then there exist $C>0$ and $0<\delta<1$ such that
\begin{equation*}
\|v\|_{L^2(B_{r_2}(Q))} \leq C \|v\|^{\delta}_{L^2(B_{r_1}(Q))} \|v\|^{1-\delta}_{L^2(B_{r_3}(Q))}
\end{equation*}
with
\[
C=C_1 e^{C_2 r_3^{-s}},
\]
where $C_1, C_2$ and $\delta$ depend on $\lambda, r_0, M_0, K_0, \frac{r_1}{r_3}, \frac{r_2}{r_3}, \textnormal{diam}(\Omega)$.
\end{proposition}

In the following proposition, we give a propagation of smallness result for \eqref{eqn: cond complex} needed for the proof of Proposition \ref{prop: uc}. A similar result can be found in \cite[Proposition 5.2]{Fos23}.

\begin{proposition}\label{prop: uc1}
For $k\in \{1,\dots,N\}$, let $v\in H^1(\mathcal{W}_k)$ be a weak solution to
\[
\textnormal{div}(\sigma \nabla v) = 0, \qquad \textnormal{in }\mathcal{W}_k,
\]
where $\sigma$ is either equal to $\sigma^{(1)}$ or $\sigma^{(2)}$. Assume that for given real numbers $\ep_0, E_0>0$, $v$ satisfies the following bounds:
\begin{equation}
\|v\|_{L^{\infty}((D_0)_{r_0\slash 2})} \leq \ep_0,
\end{equation}
with $(D_0)_{r_0\slash 2}$ as in \eqref{K00} and
\begin{equation}
|v(x)| \leq (E_0+\ep_0) (\textnormal{dist}(x,\mathcal{U}_k))^{-\tilde\gamma}, \qquad \text{for any } x\in \mathcal{W}_k,
\end{equation}
where $\tilde \gamma = n\slash 2-1$. Then the following inequality holds:
\begin{equation}\label{eqn: prop6-6}
|v(\tilde x)| \leq C r^{-\tilde\gamma} (\ep_0 +E) \Big(\frac{\ep_0}{\ep_0 + E}\Big)^{\tau_r \beta^{N_1}},
\end{equation}
where:
\begin{itemize}
    \item $C$, $N_1\in \N$ are positive constants depending on $r_0, L, \lambda, \bar{A}$;
    \item $\tilde x = P_{k+1} + r\nu(P_{k+1})$, $P_{k+1}\in (\Sigma_{k+1})_{r_0/3}$ and $\nu(P_{k+1})$ is the outward unit normal of $\p \Omega_{k+1}$ at the point $P_{k+1}$;
    \item $\beta, \tau_r \in (0,1)$.
\end{itemize}
\end{proposition}

\begin{proof}[Proof of Proposition \ref{prop: uc1}]
The proof follows the lines of \cite[Proposition 3.3]{Fos23} with minor adaptation to the conductivity case.
\end{proof}

\begin{proof}[Proof of Proposition \ref{prop: uc}]
Recall that
\[
\ep=\|\Lambda_1-\Lambda_2\|_*\quad \text{and}\quad E=\|\gamma^{(1)}-\gamma^{(2)}\|_{L^{\infty}(\Omega)}.
\]
Fix $z\in (D_0)_{r_0\slash 2}$ and set $v(y)= S_k(y,z)$ for $y\in \mathcal{W}_k$. Then, $v$ is a weak solution to 
\[
\textnormal{div}_y(\sigma^{(1)}\nabla_y v) = 0, \qquad \textnormal{in }\mathcal{W}_k.
\]
By the Alessandrini identity \eqref{eqn: alessandrini} we derive a bound for the $L^{\infty}$-norm of $v$ in a subset of $D_0$:
\[
\|v\|_{L^{\infty}((D_0)_{r_0\slash 2})} \leq C\ep,
\]
and by \eqref{eqn: Skupper}, a pointwise bound for $v$ in $\mathcal{W}_k$:
\[
|v(y)| \leq C\,E\,(\textnormal{dist}(y,\Sigma_{k+1}))^{1-n\slash 2},\qquad \textnormal{for any }y\in \mathcal{W}_k.
\]
Since the hypotheses of Proposition \ref{prop: uc1} are fullfilled, we can apply \eqref{eqn: prop6-6} at the point $\tilde x = P_{k+1} + r\nu(P_{k+1})$ and derive the estimate
\[
|S_k(\tilde x, z)| \leq C r^{-\tilde\gamma} (\ep_0 +E) \Big(\frac{\ep_0}{\ep_0 + E}\Big)^{\tau_r \beta^{N_1}}.
\]
Now, set $\tilde v(z) = S_k(\tilde x, z)$ for $z\in \mathcal{W}_k$. Hence, $\tilde v$ is a weak solution to 
\[
\textnormal{div}_z(\sigma^{(2)}\nabla_z \tilde v) = 0, \qquad \textnormal{in }\mathcal{W}_k.
\] 
By applying again Proposition \ref{prop: uc1} to $\tilde v$, we derive
then
\[
|S_k(\tilde x, \tilde x)| \leq C r^{-2 \tilde\gamma} (\ep_0 +E) \Big(\frac{\ep_0}{\ep_0 + E}\Big)^{\tau_r^{2} \beta^{2 N_1}}.
\]
For the  interior estimate of the gradient of $S_k$ \eqref{eqn: singular 22}, the argument follows the lines of \cite[Proposition 3.3]{Fos23}, but for the interior estimates in the $L^{\infty}$-norm, we refer to \cite[Theorem 1, Section 4]{DouNir55}.
\end{proof}

\section*{Acknowledgments}

The work of ES was supported by PRIN 2022 n. 2022B32J5C funded by MUR, Italy, and by the European Union – Next Generation EU. ES has also been supported by Gruppo Nazionale per l'Analisi \text{Matematica,} la Probabilità e le loro applicazioni (GNAMPA) by the grant "Problemi inversi per equazioni alle derivate parziali" CUP\_E53C22001930001. SF research was also funded by the Austrian Science Fund (FWF) SFB 10.55776/F68 "Tomography Across the Scales", project F68-01. The work of RG was partly supported by Science Foundation Ireland under grant number 16/RC/3918. RG would like to thank the Isaac Newton Institute for mathematical Sciences, Cambridge, for support and hospitality during the programme Rich and Nonlinear Tomography - a multidisciplinary approach, where work on this paper was partly undertaken. This work was supported by EPSRC grant EP/R014604/1. The work of ES was supported by PRIN 2022 n. 2022B32J5C "Inverse Problems in PDE: theoretical and numerical analysis" funded by MUR, Italy, and by the European Union Next Generation EU.



\begin{thebibliography}{10}
	
\bibitem{AdGabLio15}
{\sc A.~Adler, R.~Gaburro and B.~W.~R.~Lionheart}, {\em Electrical Impedance Tomography}, Handbook of Mathematical Methods in Imaging, ed. O. Scherzer, Springer, New York, NY, second edition, (2015), pp.~701--762.

\bibitem{DouNir55}
{\sc A.~Douglis, and L.~Nirenberg}, {\em Interior estimates for elliptic systems of partial differential equations}, Comm. Pure Appl. Math. 8 (1955), pp.~503--538. 

	
\bibitem{A1} 
{\sc G.~Alessandrini}, {\em Stable determination of conductivity by boundary measurements}, Appl. Anal. 27 (1988), pp.~153--172.

\bibitem{A} 
{\sc G.~Alessandrini}, {\em Singular solutions of elliptic equations and the determination of conductivity by boundary measurements}, J. Differential Equations 84 (1990), pp. 252--272.

\bibitem{A-dH-F-G-S} 
{\sc G.~Alessandrini, M.~V.~de~Hoop, F.~Faucher, R.~Gaburro and E.~Sincich}, {\em Inverse problem for the Helmholtz equation with Cauchy data: reconstruction with conditional well-posedness driven iterative regularization}, ESAIM: M2AN \textbf{53} (2019), 1005 - 1030.

\bibitem{Al-dH-G} 
{\sc G.~Alessandrini, M.~De~Hoop and R.~Gaburro}, {\em Uniqueness for the electrostatic inverse boundary value problem with piecewise constant anisotropic conductivities}, Inverse Problems 33 (2018), 125013.

\bibitem{AleDeHGabSin17}
{\sc G.~Alessandrini, M.~V. de~Hoop, R.~Gaburro, and E.~Sincich}, {\em Lipschitz stability for the electrostatic inverse boundary value problem with piecewise linear conductivities}, J. Math. Pures Appl. 107 (2016), pp.~638--664.
	
\bibitem{AleDeHGabSin18}
{\sc G.~Alessandrini, M.~V.~de~Hoop, R.~Gaburro and E.~Sincich}, {\em EIT in a layered anisotropic medium}, Inverse Problems and Imaging 12 (2018), pp.~667--677. 

\bibitem{A-dH-G-S1} 
{\sc G.~Alessandrini, M.~V.~de~Hoop, R.~Gaburro and E.~Sincich}, {\em Lipschitz stability for a piecewise linear Schr\"odinger potential from local Cauchy data},  Asympt. Anal. 108 (2018), 115-149.
			
\bibitem{AleDiCMorRos14}
{\sc G.~Alessandrini, M.~Di~Cristo, A.~Morassi, and E.~Rosset}, {\em Stable determination of an inclusion in an elastic body by boundary measurements}, SIAM J. Math. Anal. 46 (2014), pp.~2692--2729.

\bibitem{A-G1} 
{\sc G.~Alessandrini and R.~Gaburro},
{\em The local Calder\'{o}n problem and the determination at the boundary of the conductivity}, Comm. Partial Differential Equations 34 (2009), pp.~918-936.

\bibitem{A-G2}
{\sc G.~Alessandrini, and R.~Gaburro}, {\em The local Calderón problem and the determination at the boundary of the conductivity}, Comm. Partial Differential Equations 34 (2009), pp.~918--936.

\bibitem{AGS}
{\sc G.~Alessandrini, R.~Gaburro, and E.~Sincich},
{\em Determining an anisotropic conductivity by boundary measurements: stability at the boundary}, J. Differential Equations 382 (2024),
pp.~115--140.
	
\bibitem{AleVes05}
{\sc G.~Alessandrini and S.~Vessella}, {\em Lipschitz stability for the inverse conductivity problem}, Adv. in Appl. Math. 35 (2005), pp.~207--241.
	
\bibitem{B-B-R} 
{\sc J.~A.~Barcel\'{o}, T.~Barcel\'{o} and A.~Ruiz}, {\em Stability of the inverse conductivity problem in the plane for less regular conductivities}, J. Differential Equations 173 (2001), pp.~231--270.
	
\bibitem{B-F-R} 
{\sc T.~Barcel\'{o}, D.~Faraco and A.~Ruiz}, {\em Stability of Calder\'{o}n's inverse conductivity problem in the plane}, J. Math. Pures Appl. 88 (2007), pp.~522--556.
	
\bibitem{BerFra11}
{\sc E.~Beretta and E.~Francini}, {\em Lipschitz stability for the electrical impedance tomography problem: the complex case}, Comm. Partial Differential Equations 36 (2011), pp.~1723--1749.

\bibitem{Be-dH-Q} 
{\sc E.~Beretta, M.~De~Hoop and L.~Qiu}, {\em Lipschitz stability of an inverse boundary value problem for a Schr\"{o}dinger type equation}, SIAM J. Math. Anal. 45 (2013), pp.679--699.
	
\bibitem{Be-dH-F-S} 
{\sc E.~Beretta, M.~De~Hoop, F.~Faucher and O.~Scherzer}, {\em Inverse boundary value problem for the Helmholtz equation: quantitative conditional Lipschitz stability estimates}, SIAM J. Math. Anal. 48 (2016), pp.~3962--3983.

\bibitem{Be-Fr-V} 
{\sc E.~Beretta, E.~Francini and S.~Vessella}, {\em Uniqueness and Lipschitz stability for the identification of Lam\'e parameters from boundary measurements}, Inv. Probl. Imag. 8 (2014), pp.~611--644.
	
\bibitem{Be-Fr-Mo-Ro-Ve} 
{\sc E.~Beretta, E.~Francini, A.~Morassi, E.~Rosset and S.~Vessella}, {\em Lipschitz continuous dependence of piecewise constant Lam\'e coefficients from boundary data: the case of non flat interfaces}, Inverse Problems 30 (2014), 125005.
	
\bibitem{Bo}
{\sc L.~Borcea}, {\em Electrical impedance tomography}, Inverse Problems 18 (2002), pp.~R99--R136.

\bibitem{Bu} 
{\sc A.~L.~Buckgeim}, {\em Recovering a potential from Cauchy data in the two-dimensional case}, J. Inverse Ill-Posed Probl. 16 (2008), pp.~19--33.
	
\bibitem{C} 
{\sc A.~P.~Calder\'{o}n}, {\em On an inverse boundary value problem}, 
Seminar on Numerical Analysis and its Applications to Continuum Physics (Rio de Janeiro, 1980),   65--73, Soc. Brasil. Mat., Rio de Janeiro, 1980. Reprinted in: Comput. Appl. Math. 25 (2006), pp.~133--138.
	
\bibitem{CarNguWan20}
{\sc C.~I.~C\^{a}rstea, T.~Nguyen, and J.-N.~Wang}, {\em Uniqueness estimates for the general complex conductivity equation and their applications to inverse problems}, SIAM J. Math. Anal. 52 (2020), pp.~570--580.
		
\bibitem{D1} 
{\sc V.~Druskin}, {\em The unique solution of the inverse problem of electrical surveying and electrical well-logging for piecewise-continuous conductivity}, Izv. Earth Phys. 18 (1982), pp.~51--53 (in Russian).
	
\bibitem{D2} 
{\sc V.~Druskin}, {\em On uniqueness of the determination of the three-dimensional underground structures from surface measurements with variously positioned steady-state or monochromatic field sources}, Sov. Phys.-Solid Earth 21 (1985), pp.~210-214 (in Russian).
	
\bibitem{D3} 
{\sc V.~Druskin}, {\em On the uniqueness of inverse problems from incomplete boundary data}, SIAM J. Appl. Math. 58 (1998), pp.~1591--1603.
	
\bibitem{dH-Q-S} 
{\sc M.~De~Hoop, L.~Qiu and O.~Scherzer}, {\em Local analysis of inverse problems: H\"older stability and iterative reconstruction}, Inverse Problems 28 (2012), 045001.
	
\bibitem{dH-Q-S1} 
{\sc M.~De~Hoop, L.~Qiu and O.~Scherzer}, {\em An analysis of a multi-level projected steepest descent iteration for nonlinear inverse problems in Banach spaces subject to stability constraints}, Numerische Matematik 129 (2015), pp.~127--148.

\bibitem{F-A-Ba-dH-G-S} 
{\sc F.~Faucher, G.~Alessandrini, H.~Barucq, M.~V.~de~Hoop, R.~Gaburro and E.~Sincich}, {\em Full reciprocity-gap waveform inversion enabling sparce-source acquisition}, Geophysics 85 (2020), pp.~R461--R476.

\bibitem{Fa-dH-S} 
{\sc F.~Faucher, M.~V.~de~Hoop and O.~Scherzer}, {\em Reciprocity-gap misfit functional for distributed acoustic sensing, combining data from passive and active sources}, Geophysics 86 (2021).

\bibitem{Fos23}
{\sc S.~Foschiatti}, {\em Lipschitz stability estimate for the simultaneous recovery of two coefficients in the anisotropic Schrödinger type equation via local Cauchy data},
Journal of Mathematical Analysis and Applications 531 (2024), 127753.

\bibitem{FosGabSin21}
{\sc S.~Foschiatti, R.~Gaburro, and E.~Sincich}, {\em Stability for the {C}alder\'{o}n's problem for a class of anisotropic conductivities via an ad hoc misfit functional}, Inverse Problems 37 (2021), 125007.

\bibitem{Fr} 
{\sc E.~Francini}, {\em Recovering a complex coefficient in a planar domain from the Dirichlet-to-Neumann map}, Inverse Problems 16 (107) (2000), pp.~107-119.
	
\bibitem{FraVesWan23}
{\sc E.~Francini, S.~Vessella, and J.-N. Wang}, {\em Propagation of smallness and size estimate in the second order elliptic equation with discontinuous complex {L}ipschitz conductivity}, J. Differential Equations 343 (2023), pp.~687--717.
	
\bibitem{G-S} 
{\sc R.~Gaburro and E.~Sincich}, {\em Lipschitz stability for the inverse conductivity problem for a conformal class of anisotropic conductivities}, Inverse Problems 31 (2015), 015008.

\bibitem{Ga-H} 
{\sc H.~Garde and N.~Hyv\"onen}, {\em Optimal depth-dependent distinguishability bounds for Electrical Impedance Tomography in arbitrary dimension}, SIAM J. Math. Anal. 80 (2020), pp.~20-43.
	
\bibitem{HofKim07}
{\sc S.~Hofmann and S.~Kim}, {\em The {G}reen function estimates for strongly elliptic systems of second order}, Manuscripta Math. 124 (2007), pp.~139--172.
	
\bibitem{KanKim10}
{\sc K.~Kang and S.~Kim}, {\em Global pointwise estimates for {G}reen's matrix of second order elliptic systems}, J. Differential Equations 249 (2010), pp.~2643--2662.
	
\bibitem{Kir05}
{\sc A.~Kirsch}, {\em The factorization method for a class of inverse elliptic problems}, Mathematische Nachrichten 278 (2005), pp.~258--277.

\bibitem{Koh-V1} 
{\sc R.~Kohn and M.~Vogelius}, {\em Identification of an unknown conductivity by means of measurements at the boundary}, SIAM-AMS Proc. 14 (1984), pp.~113--123.

\bibitem{Koh-V2}
{\sc R.~Kohn and M.~Vogelius}, {\em Determining conductivity by boundary measurements II. Interior Results}, Comm. Pure Appl. Math., 38 (1985), pp.~643--667.

\bibitem{La-U} 
{\sc M.~Lassas and G.~Uhlmann}, {\em On determining a Riemannian manifold from the Dirichlet-to-Neumann map}, Ann. Sci. \'Ecole Norm. Sup. 34 (2001), No. 5, pp.~771--787.
	
\bibitem{La-U-T} 
{\sc M.~Lassas, G.~Uhlmann and M.~Taylor}, {\em The Dirichlet-to-Neumann map for complete Riemannian manifolds with boundary},  Comm. Anal. Geom. 11 (2003), pp.~207--221.	

\bibitem{Le-U} 
{\sc J.~M.~Lee and G.~Uhlmann}, {\em Determining anisotropic real-analytic conductivities by boundary measurements}, Comm. Pure Appl. Math., 42 (1989), pp.~1097--1112.
		
\bibitem{LiNir03}
{\sc Y.~Li and L.~Nirenberg}, {\em Estimates for elliptic systems from composite material}, Comm. Pure Appl. Math. 56 (2003), pp.~892--925. \newblock Dedicated to the memory of J\"{u}rgen K. Moser.

\bibitem{Lio}
{\sc W. R. B. Lionheart}, {\em Conformal uniqueness results in anisotropic electrical impedance imaging} Inverse Probl. 13 (1997),
pp.~125--134.
	
\bibitem{Liu} 
{\sc L.~Liu}, {\em Stability estimates for the two-dimensional inverse conductivity problem}, PhD Thesis, University of Rochester, New York, (1997).
 
\bibitem{Ma} 
{\sc N.~Mandache}, {\em Exponential instability in an inverse problem for the Schr\"{o}dinger equation}, Inverse Problems 17 (2001), pp.~1435--1444.
	
\bibitem{N} 
{\sc A.~Nachman}, {\em Global Uniqueness for a two-dimensional inverse boundary value problem}, Ann. Math. 143 (1996), pp.~71--96.
	
\bibitem{R} 
{\sc L.~Rondi}, {\em A remark on a paper by G. Alessandrini and S. Vessella: "Lipschitz stability for the inverse conductivity problem"}, Adv. Appl. Math. 36 (2006), pp.~67--69.
	
\bibitem{RS1} 
{\sc A.~R{\"u}land and E.~Sincich}, {\em Lipschitz stability for the finite dimensional fractional Calder{\'o}n problem with finite Cauchy data},
Inv. Probl. Imag. 13 (2019), pp.~1023--1044.
	
\bibitem{RS2}
{\sc A.~R{\"u}land and E.~Sincich}, {\em On Runge approximation and Lipschitz stability for a finite-dimensional Schrödinger inverse problem}, Appl. Anal. 101 (2020), pp.~3655--3666.

	
\bibitem{Sy-U} 
{\sc J.~Sylvester and G.~Uhlmann},
{\em A global uniqueness theorem for an inverse boundary valued problem}, Ann. Math. 125 (1987), pp.~153--169.
	
\bibitem{U} 
{\sc G.~Uhlmann}, {\em Electrical impedance tomography and Calder\'{o}n's problem
(topical review)}, Inverse Problems 25 (2009), 123011.	
	
\end{thebibliography}
\end{document}